\documentclass[11pt]{amsart}

\usepackage{algorithm}%
\usepackage{algorithmicx}%
\usepackage{algpseudocode}%
\usepackage{hyperref,url}
\usepackage[numbers,sort&compress]{natbib} 
\usepackage{amsfonts,bbm}
\usepackage{mathrsfs} 
\usepackage{appendix}
\usepackage{mathtools}
\mathtoolsset{showonlyrefs}
\usepackage{amsmath,amssymb,amsthm,bbm}
\usepackage{latexsym}
\RequirePackage{natbib}
\usepackage{listings}%

\newtheorem{theorem}{Theorem}[section]
\newtheorem{lemma}[theorem]{Lemma}
\newtheorem{corollary}[theorem]{Corollary}
\newtheorem{proposition}[theorem]{Proposition}

\newtheorem{remark}[theorem]{Remark}

\newtheorem{assumption}[theorem]{Assumption}

\newcommand*{\QEDA}{\hfill\ensuremath{\square}}
\def\RR{{\mathbb R}}  

\def\NN{{\mathbb N}}      
\def\Fix{\mathop{\rm Fix}}      
\newcommand{\norm}[1]{\left\Vert #1 \right\Vert} 
\DeclareMathOperator{\dist}{dist}
\DeclareMathOperator{\spann}{span}
\DeclareMathOperator{\prog}{prog}
\DeclareMathOperator{\argmax}{argmax}

\def\A{\mathcal{A}}
\def\S{\mathcal{S}}
\def\T{\mathcal{T}}
\def\H{\mathcal{H}}
\def\EE{{\mathbb E}}        
\def\FF{{\mathcal F}}  
\def\dim{{|\S\times\A|}}  
\def\tm{{t_{\text{mix}}}}  
\def\hsp{{\norm{h^*}_{\text{sp}}}}
\def\hspp{{\norm{h^*}^7_{\text{sp}}}}

\usepackage{graphicx}%
\usepackage{multirow}%
\usepackage{amsmath,amssymb,amsfonts}%
\usepackage{amsthm}%
\usepackage{booktabs}%
\usepackage{listings}%
\usepackage{verbatim}
\usepackage[textwidth=30mm]{todonotes}

\begin{document}
%%%%%%%%%%%%%%%%%%%%%%%%%%%%%%%%%%%%%%%%%%%%%%%%%%%%%%%%%%%%%%%%%
%%%%%%%%%%%%%%%%%%%%%%%%%%%%%%%%%%%%%%%%%%%%%%%%%%%%%%%%%%%%%%%%%

\title[Stochastic Halpern and applications]{Stochastic Halpern iteration in normed spaces and applications to reinforcement learning}

\author[M. Bravo]{Mario Bravo}
\address[M.B.]{Universidad Adolfo Ib\'a\~nez, Facultad de Ingenier\'ia y Ciencias,
Diagonal Las Torres 2640, Santiago, Chile} 
\email{ \href{mailto:mario.bravo@uai.cl}{\nolinkurl{mario.bravo@uai.cl}}}

\author[J.P. Contreras]{Juan Pablo Contreras}
\address[J.P.C.]{Escuela de Ingenier\'ia Industrial, Universidad Diego Portales,
Av. Ej\'ercito Libertador 441, Santiago, Chile} 
\email{ \href{mailto:juan.contreras@udp.cl}{\nolinkurl{juan.contreras@udp.cl}}}

\begin{abstract}We analyze the oracle complexity of the stochastic Halpern iteration with minibatch, where we aim to approximate fixed-points of nonexpansive and contractive operators in a normed finite-dimensional space. We show that if the underlying stochastic oracle has uniformly bounded variance, our method exhibits an overall oracle complexity of $\tilde{O}(\varepsilon^{-5})$, to obtain $\varepsilon$ expected fixed-point residual for nonexpansive operators, improving recent rates established for the stochastic Krasnoselskii-Mann iteration. Also, we establish a lower bound of $\Omega(\varepsilon^{-3})$ which applies to a wide range of algorithms, including all averaged iterations even with minibatching. Using a suitable modification of our approach, we derive a $O(\varepsilon^{-2}(1-\gamma)^{-3})$ complexity bound in the case in which the operator is a $\gamma$-contraction to obtain an approximation of the fixed-point. As an application, we propose new model-free algorithms for average and discounted reward MDPs. For the average reward case, our method applies to weakly communicating MDPs without requiring prior parameter knowledge.\end{abstract}



\maketitle

\section{Introduction}
Fixed-point iterations are fundamental in mathematics and engineering, providing a systematic way to approximate the solutions of several problems, including operator equations, nonlinear equations, and optimization problems. These iterations form the backbone of many optimization algorithms, such as the gradient descent or splitting methods, which are widely used in machine learning, control systems, and scientific computing.

In recent years, stochastic variants of these problems have gained significant attention because of their applicability in scenarios where the underlying operators are not known exactly or are subject to noise, uncertainty, or randomness. Nevertheless, general stochastic fixed-point iterations have received little attention compared to other related setups, such as stochastic monotone inclusions, stochastic variational inequalities, or stochastic optimization which are mainly developed in the Euclidean setting. 

The main goal of this work is then to study a simple stochastic version of the well-known Halpern iteration, where we assume that the underlying operator is nonexpansive or contracting for some general norm in finite dimension. Using a simple approach derived mainly from the construction of recursive bounds and a simple minibatching technique, we study the overall oracle complexity of this algorithm while showing applications that are intrinsically of a non-Euclidean nature.

\subsection{Problem setup}
We consider $\RR^d$ endowed with a general norm $\norm{\,\cdot\,}$, and denote by $\norm{x}_2$ the standard Euclidean norm, {\em i.e.} $\norm{x}_2^2 = \sum_{i=1}^d x_i^2$. In what follows,  $T:\RR^d \to \RR^d$ is such that
\[
\norm{Tx-Ty} \le \gamma\norm{x-y}, \qquad \forall x,y\in \RR^d.
\]
We say that $T$ is a contraction if $\gamma \in (0, 1)$, and that $T$ is nonexpansive if $\gamma = 1$. We assume that in the nonexpansive case $\Fix \, T \neq \emptyset$, where $\Fix\, T$ stands for the set of fixed-points of $T$. We are interested in computing an approximate solution of the fixed-point equation $x=Tx$.

Our method is primarly inspired by the standard Halpern iteration \cite{halpern1967fixed}
\begin{equation} \tag{$\textsc{H}$}\label{halpern_noiseless}
x^{n} = (1-\beta_{n})x^0 + \beta_{n}Tx^{n-1},  \quad \forall \, n\ge 1,
\end{equation}
where $(\beta_n)_n$ is a nondecreasing stepsize sequence satisfying $0< \beta_n <1$, $\lim_{n \to \infty}\beta_n=1$, and $x^0 \in \RR^d$ is an arbitrary initial condition.  It is worth mentioning that the original Halpern iteration uses an arbitrary $u$ called {\em the anchoring point} that can be different from the initial point $x^0$. In that case, the iteration reads $x^{n} = (1-\beta_{n})u + \beta_{n}Tx^{n-1}$. We assume here, without loss of generality, that the anchoring point is exactly $x^0$.

Naturally, iteration \eqref{halpern_noiseless} is only implementable in the case where one has access to $Tx$ at any point $x$.  In this paper, we assume that the vector $Tx$ is only accessible through $\tilde{T}$ which is a noisy version of $T$. Specifically, we consider the following
\begin{assumption} \label{assumption1} 
    The access to the operator $T$ occurs through a stochastic oracle $\tilde{T}: \RR^d\times \Xi \to \RR^d$ which is an unbiased approximation of $T$ in the sense that 
\begin{equation*}
    \EE(\tilde{T}(x,\xi)) = Tx \qquad \mbox{for all $x\in \RR^d$},
\end{equation*}
with $\xi$ drawn from some distribution $\mathcal D_x$ over $\Xi$. Additionally, we assume that $\EE \left(||\tilde{T}(x,\xi)-Tx||^2_2\right)$ is finite for every $x \in \RR^d$.
\end{assumption}
We say that the oracle $\tilde T$ has {\em uniformly bounded variance} (with bound $\sigma^2)$ if there exists a positive constant $\sigma^2$ such that
\begin{equation*}
 \sup_{x \in \RR^d}\EE \left(\norm{\tilde{T}(x,\xi)-Tx}^2_2\right)\leq \sigma^2.
\end{equation*}
Uniformly bounded variance is evidently a stronger condition than the variance requirement in Assumption~\ref{assumption1}. Importantly, our results extend beyond the uniformly bounded case, permitting the variance to grow in a controlled manner with the state variable $x^n$. This flexibility is crucial for the application to MDPs discussed in Section~\ref{sec:MDP}.

A first possible idea is to consider the iteration 
$x^n= (1-\beta_{n})x^0 + \beta_{n}\, \tilde Tx^{n-1}$. Intuitively, even if the term $\tilde{T}(x,\xi)-Tx$ is with uniformly bounded variance, because $\beta_n$ goes to 1, the last iterate will inevitably contain some nonnegligible noise term. To address this issue, we adopt a simple increasing minibatch strategy designed to mitigate the variance of the error. This approach lead to the study of the following inexact version of the Halpern iteration \eqref{halpern_noiseless}, 
\begin{equation}
    x^{n} = (1-\beta_{n})x^0 + \beta_{n}(Tx^{n-1} + U_{n}), \quad \forall \, n \geq 1,
    \label{halpern}
    \tag{$\textsc{iH}$}
\end{equation}
where $U_n$ is such that $\EE(U_n)=0$ and the term $\EE(\norm{U_n})$ can be controlled appropriately. It is worth noting that we do not assume further properties on the stochastic oracle in order to maintain our results as general as possible. See Section~\ref{sec:method} for a precise definition and related discussions. 

We are interested in the problem of estimating the oracle complexity to find an approximate fixed-point of $T$, which is measured as the number of calls to the vector stochastic oracle $\tilde T$. In the case where $T$ is nonexpansive, we estimate the oracle complexity to find a last-iterate approximate solution $x^n$ in the sense that the expected fixed-point residual $\EE (\norm{x^n - Tx^n})$ is smaller than a given tolerance. For a contracting operator $T$, we aim to find oracle complexity bounds for the expected distance of the last iterate $x^n$ to the unique fixed-point $x^*$, {\em i.e.} the quantity $\EE (\norm{x^n - x^*})$.

\vspace{1ex}
\noindent {\bf Notation.} Let us set some basic notation that we use in this paper.  Given that all norms in $\RR^d$ are equivalent, let $\mu > 0$ be such that $\norm{x} \le \mu \norm{x}_2$ for all $x\in \RR^d$ and we denote by $\norm{x}_{\text{sp}}=\max_{i}x_i - \min_{i}x_i$ the so-called span seminorm of $x$.

We adopt the standard big-$O$ notation, that is, for two real functions $f,g:$ we say $f(x)=O(g(x))$ if there are two constants $C$ and $c$ such that $f(x)\leq C g(x)$ for all $x\geq c$. We use $\tilde O$ if we ignore logarithmic factors. Analogously, $f(x)=\Omega(g(x))$ if $f(x)\geq C g(x)$ for all $x\geq c$.

\subsection{Our contributions}
In what follows, we briefly describe our main findings.

\noindent {\bf Nonexpansive maps.} We show that if $T$ is nonexpansive and the stochastic oracle is with bounded variance, then our method achieves last-iterate oracle complexity of $\tilde{O}(\varepsilon^{-5})$ for oracles with uniformly bounded variance (see Corollary~\ref{cor:bounded_variance}). This result improves over the guarantee $\tilde{O}(\varepsilon^{-6})$ established for the stochastic Krasnoselskii-Mann iteration in \cite{bravo2024stochastic} also for the last-iterate. This is, to the best of our knowledge, the only general complexity result available at the moment for the non-Euclidean setting.  

Additionally, we establish a lower bound of $\Omega(\varepsilon^{-3})$ for the overall oracle complexity for nonexpansive maps with uniformly bounded variance stochastic oracle and where the underlying operator has bounded range. In particular, this lower bound applies to every algorithm such that the iterations are a linear combination of the previous  oracle calls, including well-known algorithms such as the Krasnoselskii-Mann and Halpern iteration. Our result is also true in the case where the use of minibatches is allowed ({\em c.f.} Theorem~\ref{thm:lowerbound}). 

Finally, we extend classical results in \cite{xu2002iterative} by proving that the iteration converges to a fixed point of $T$ provided that the norm of the space is smooth ({\em c.f.} Theorem~\ref{thm:convergence}).

\noindent {\bf Contractions.} In the case where $T$ is a contraction with parameter $\gamma\in (0,1)$, we prove that if the stochastic oracle is with uniformly bounded variance, then our method finds an $\varepsilon-$optimal solution after at most $O\left(\varepsilon^{-2}(1-\gamma)^{-3}\right)$ queries (Corollary~\ref{col:uniform-contractive}). Given that this is a general result, the bound depends on the specific instance and can be refined if more structure is available. Indeed, the same holds for the nonexpansive case.

Nevertheless, it is worth noting that, as pointed out in recent work \cite{mou2022optimal}, little attention has been devoted to the study of stochastic fixed-point iterations for contracting operators without assuming some form of monotonicity for $T$ associated with an underlying Bellman operator (see Remark~\ref{rem:monotone} for details). Up to our knowledge, there is no result available in this generality showing explicit dependence on $(1-\gamma)$ and $\varepsilon$.

\noindent {\bf Markov Decision Processes (MDPs).} 
As an application of our approach, we present a new model-free algorithm for weakly communicating MDPs with average reward, assuming the use of a generative model that can sample the next state from any state-action pair.
In Theorem~\ref{teo_average} 
we show that, given  $\varepsilon > 0$ Algorithm~\ref{alg_av} can find an iterate such that the Bellman error is at most $\varepsilon$ with a sample complexity of $\tilde{O}(\dim \hspp \varepsilon^{-7})$, where $\mathcal{S}$ and $\mathcal{A}$ represent the state and action spaces, respectively, and $h^*$ denotes an optimal bias function. Our result on the Bellman error can be used to find an $\varepsilon$-optimal policy in expectation and with high probability with the same sample complexity, (see Corollary~\ref{cor:policy_error}). Our procedure is without prior knowledge in the sense that it requires no parameter information as input but it does not provide a stopping rule to achieve $\varepsilon$ precision, as it is the case for other related algorithms (see Section~\ref{sec:related} for details). We then extend our results to the discounted case, where the induced procedure outputs a $Q$ value with expected distance $\varepsilon>0$ to the optimal $Q^*$ after at most $\tilde{O}\left(\dim \varepsilon^{-2}(1-\gamma)^{-5}\right)$ samples (see Theorem~\ref{teo_discounted}). This matches some well-known results in the literature (see e.g. \cite{wainwright2019stochastic}) in the context of model-free procedures. 

Finally, we emphasize that, in both the average and discounted settings, our analysis does not rely on specific structural properties of the Bellman operator, such as component-wise monotonicity (see Remark~\ref{rem:monotone} for a discussion).

\noindent {\bf Our Approach.} The techniques we use to obtain our results are mainly based on a recursive bound approach developed first to the study of the well-known Krasnoselskii-Mann iteration,
\begin{equation}\label{km} \tag{$\textsc{KM}$}
x^{n}=(1\!-\!\alpha_{n})x^n+\alpha_{n} Tx^{n-1}\quad n\geq 1
\end{equation}
in Banach spaces, where $(\alpha_n)_n \subseteq (0,1)$ is a sequence of stepsizes \cite{baillon1996rate, cominetti2014rate, bravo2018sharp}. Later, the analysis was extended to general averaged fixed-point iteration \cite{bravo2021universal}. In the case of the Halpern iteration, being a particular case of an averaged iteration, the bounds developed are particularly explicit, giving way to an analysis that is simple yet effective. Roughly speaking, in the stochastic setting, we will adapt these ideas by compensating the effect of repeated calls to the oracle with the smooth contracting effect of the term $\beta_n<1$. Given that $\beta_n$ approaches one, we will take increasing minibatch sizes. It is worth mentioning that our explicit results are obtained for the case $\beta_n=\frac{n}{n+1}$. The reason behind this particular choice is twofold. On the one hand, it is known that in the noiseless case this stepsize gives the best possible bounds in the Euclidean setting \cite{sabach2017first,lieder2021convergence} and, up to log factors, also in the Banach framework. On the other hand, as we shall see in what follows, this choice leads to clean estimates with computable constants when calculating fixed-point residuals.

\vspace{6ex}
\subsection{Related work}\label{sec:related}\hfill

\noindent  {\bf Deterministic setup.} For nonexpansive maps in general normed spaces, recent efforts have been devoted to establish nonasymptotic convergence guarantees for the fixed-point residual $\norm{x^n - Tx^n}$ of some classical iterations. For \eqref{km}, convergence rates were initially determined by \cite{baillon1992optimal,baillon1996rate} for constant stepsizes $\alpha_n\equiv\alpha$, showing a rate of $O(1/\sqrt{n})$. Later, these results were extended by \cite{cominetti2014rate} to arbitrary steps in $(0, 1)$. The sharpness of the bound in \cite{cominetti2014rate} was obtained in \cite{bravo2018sharp}. The first proof of the rate $O(1/n)$ was established in \cite{sabach2017first} using the Halpern iteration \eqref{halpern_noiseless} with the specific stepsize $\beta_{n} = \frac{n}{n+2}$. Later, in \cite{bravo2021universal}, it was shown that the traditional Halpern stepsize choice $\beta_n = \frac{n}{n+1}$ only yields the rate $O(\frac{\log n}{n})$, which is tight, while in \cite{contreras2023optimal} the optimal coefficients for the Halpern iteration were determined, improving the constants involved.

For Euclidean spaces, the exact upper bounds were determined using the Performance Estimation Problems (PEP) approach, originally proposed in \cite{drori2014performance}. However, the convergence rate does not show an improvement compared to the general Banach setting. The classical Halpern iteration with stepsize $\beta_n = \frac{n}{n+1}$ was examined in \cite{lieder2021convergence} and the author proved the convergence rate $O(1/n)$, which improves by a logarithmic factor compared to the convergence rate in general normed spaces.

Halpern iteration has also been successfully applied in the context of monotone inclusions, which is strongly connected to fixed-point iterations \cite{bauschke2011convex} in the Euclidean setup. In particular, using PEP, \cite{kim2021accelerated} determined an optimal inertial iteration to approximate the zeros of a cocoercive operator that was shown to be equivalent to the Halpern iteration in \cite{contreras2023optimal} and \cite[Chapter 12.2]{ryu2021large}. Beyond PEP, convergence rates have also been established using a potential function approach \cite{diakonikolas2020halpern, diakonikolas2021potential, tran2021halpern, lee2021fast}.

\noindent  {\bf Stochastic setup.}  In the Euclidean setting, Halpern-type iterations have been applied to approximate solutions of Minty variational inequalities for Lipschitz operators satisfying the negative comonotone property \cite{lee2021fast}, as well as to find the zeros of a cocoercive operator \cite{cai2022stochastic}. The latter approach employs minibatches with linear growth, leveraging techniques developed for the deterministic setting of nonexpansive operators \cite{diakonikolas2020halpern}. This strategy leads to an overall oracle complexity of $O(\varepsilon^{-4})$ for obtaining an approximate solution in expectation within a tolerance $\varepsilon$. The authors also show that the rate $O(\varepsilon^{-3})$ is achievable under further assumptions on the stochastic oracle, such as the multi-point oracle and Lipschitz in expectation properties.  Under the same enhanced stochastic oracle, \cite{liang2024inexact} obtained similar results in the cocoercive setting. Recently, \cite{alacouglu2024} studied a stochastic Halpern iteration for cohypomonotone operators \cite{bauschke2021generalized}, a condition that allows possible nonmonotonicity while including monotone maps.  The authors establish, under no further assumptions on the oracle structure, a complexity of $O(\varepsilon^{-4})$ for the fixed-point residual.

Up to our knowledge, the only work that studies convergence rates for stochastic fixed-point iterations of nonexpansive maps in general normed spaces is \cite{bravo2024stochastic}. Here, the authors consider the stochastic Krasnoselskii-Mann iteration $x^{n}=(1\!-\!\alpha_{n})x^n+\alpha_{n} (Tx^{n-1} + U_{n})$, where $U_n$ is assumed to be a martingale difference sequence. The authors prove that if the variance of the noise is uniformly bounded, the stochastic Krasnoselskii-Mann iteration with constant stepsize requires at most $O(\varepsilon^{-6})$ iterations to reach a solution whose expected fixed-point residual is less than $\varepsilon$. This is done using a well-chosen constant stepsize.

 For contractive operators in general normed spaces, recent works have established an \( O(\varepsilon^{-2}) \) convergence rate. In particular, \cite{chen2020finite} employs a Lyapunov function-based analysis, while \cite{wainwright2019stochastic} studies the problem under monotonicity and quasi-contractivity assumptions defined with respect to a partial order and gauge norms induced by an underlying cone. The optimality of the \( O(\varepsilon^{-2}) \) rate can be justified through the connection between stochastic approximation and classical statistical theory. Under mild conditions, the sum of i.i.d. random variables in Banach spaces satisfies a central limit theorem. For a formal discussion, we refer the reader to \cite{mou2022optimal} and \cite{chen2020finite}.

\noindent  {\bf Lower bounds.} In the noiseless scenario, the lower bound for fixed-point iterations was determined in \cite{contreras2023optimal} showing that the rate $\norm{x^n-Tx^n}=O(1/n)$ is optimal for every averaged iteration. Moreover, this rate cannot be improved even in spaces with more structure, such as Hilbert spaces. In particular, \cite{park2022exact} proved that no method such that $x^{n}$ belongs to the linear span of $\{ x^0, Tx^0,\dots,Tx^{n-1}\}$ can achieve a better convergence rate. 

For stochastic convex optimization, the results date back to the information-based complexity developed in \cite{nemirovskij1983problem}.  In this line, improvements in the dimension dependence were obtained in \cite{agarwal2012information} for the setup of statistical estimation. Similar results were obtained in \cite{raginsky2011information} with a different technique that yields lower bounds for a special class of descent methods that are not covered in \cite{agarwal2012information}. Finding stationary points of convex functions relates to finding fixed-points of nonexpansive operators. In this context, the work \cite{foster2019complexity} established the lower bounds for stochastic convex optimization problems in the local and global stochastic oracle model. For nonconvex problems, the lower bound was established for the first time in \cite{fang2018spider} for the sum of smooth functions, matching the upper bounds of the variance-reduced methods SPIDER \cite{fang2018spider} and SNVRG \cite{zhou2020stochastic}. For further results and discussion, we refer the reader to \cite{drori2020complexity, arjevani2023lower} and the references therein.

\noindent {\bf Markov Decision Processes (MDPs).}  A significant part of the research on complexity for MDPs focuses on computing an approximately optimal policy or an aproximate value function, given access to a generative model. From an algorithmic perspective, there are two main approaches. The first is model-based methods, where transitions and rewards are sampled in advance in a full batch and then used to solve the model using dynamic programming techniques. In contrast, model-free algorithms estimate values directly from interactions with the environment \cite{sutton2018reinforcement}.

For the $\gamma$-discounted reward case, \cite{gheshlaghi2013minimax} establish a $\Omega(\dim (1-\gamma)^{-3}\varepsilon^{-2})$ complexity lower bound for obtaining an $\varepsilon$-optimal policy --also valid to compute almost optimal $Q$-factors. Algorithms with oracle complexity matching this lower bound has been obtained both in the model-base case (see, e.g \cite{agarwal2020model} for policies) and by means of model-free procedures (e.g. \cite{wainwright2019variance, jin2024truncated} for values and policies, respectively). The reader is referred to the discussion in \cite{jin2024truncated} for a recent account on this and related matters.

For average reward MDPs, to compute $\varepsilon$-optimal policies, \cite{jin2021towards} established an $ \Omega(\dim \tm\varepsilon^{-2})$ lower bound and designed a model-based algorithm achieving $\tilde O( \dim \tm \varepsilon^{-3})$, under the assumption that there exists a finite bound $\tm$ for the mixing time of all policies. A complexity lower bound of $ \Omega(\dim \hsp\varepsilon^{-2})$ for weakly communicating MDPs was established in \cite{wang2022near}. Recently \cite{zurek2025span} studied a model-based algorithm achieving this complexity for weakly communicating and multichain MDPs, a broader class.

Regarding model-free procedures, \cite{jin2020efficiently}  obtains a sample complexity of $\tilde O(\dim \tm\varepsilon^{-2})$, assuming finite mixing times. Using the same mixing hypothesis, \cite{li2024stochastic} proposes a procedure with a sample complexity of $O(\dim t_{\text{mix}}^3\varepsilon^{-2})$,
which is valid for unichain MDPs, a smaller class than weakly communicating.  It is worth noting that these works require prior knowledge to run the algorithm, often given by $\tm$ or $\hsp$.
A model-free algorithm without prior knowledge is studied in \cite{jin2024feasible}, where a sample complexity of $\tilde O(\dim t_{\text{mix}}^{8}\varepsilon^{-8})$ is obtained, again under a finite $t_{\text{mix}}$. Also, \cite{bravo2024stochastic} establishes a complexity of $O(\dim^{3/2}\varepsilon^{-a})$, with $a > 10$, for the RVI-$Q$-learning algorithm \cite{abounadi2001learning}, under the expected Bellman error metric $\EE(\norm{\H Q-Q}_\infty)$ for unichain MDPs, where $\H$ is the Bellman operator. We stress that in both \cite{jin2024feasible} and \cite{bravo2024stochastic}, as in our algorithm, there is no rule to stop the procedure at a given precision $\varepsilon$ unless more information is available. Recently, during the first revision of this paper and the preparation of its revised version, a near-optimal complexity of \( \tilde{O}(|\mathcal{S} \times \mathcal{A}| \, \|h_*\|_{\text{sp}}^2 \varepsilon^{-2}) \) was obtained in \cite{lbc25}. This work addresses weakly communicating average reward MDPs and proposes a model-free algorithm that does not require prior knowledge of system parameters. The method builds on a variant of the stochastic Halpern iteration presented here by using a recursive sampling technique that exploits the specific structure of the Bellman operator. Also, it implements a stopping rule to achieve the prescribed precision. Another work that implements a stopping rule without prior knowledge is \cite{tuynman24}, where a model-based algorithm is proposed through the estimation of the diameter, and therefore only being applicable to communicating MDPs.

\subsection{Organization of the paper}
The remainder of the paper is organized as follows. In Section~\ref{sec:method} we present the precise definition of our method along with the minibatch variance reduction. Section~\ref{sec:nonexpansive} is devoted to the study of the case in which $T$ is nonexpansive while in Section~\ref{sec:contracting} we study the contracting case. In Section~\ref{sec:MDP} we develop our applications to MDPs in the average and discounted setting. Finally, Appendix~\ref{app:proofs} gathers some deferred proofs.

\section{The Stochastic Halpern iteration} \label{sec:method}

Algorithm~\ref{alg_stochastic_halpern} outlines the stochastic Halpern iteration with minibatches. At each iteration $n \leq N$, a minibatch estimator is constructed with $k_n$ queries to $\tilde T$ at $x^{n-1}$ using $k_n$ independent fresh samples from $\mathcal D_{x^{n-1}}$. Using this estimate, the next iterate is calculated via the Halpern iteration.

\begin{algorithm}[H]
\caption{Stochastic Halpern method}
\label{alg_stochastic_halpern}
\begin{algorithmic}[1]
\Require Stochastic oracle $\tilde{T}$, $(k_n)_n$ minibatch size sequence, number of iterations $N$, initial point $x^0$    
\For{$n=1,\dots, N$}
\State Set the size of the minibatch $k_n$ 
\State Sample $k_n$ independent queries at $x^{n-1}$: $\tilde{T}(x^{n-1}, \xi^n_{1}), \dots, \tilde{T}(x^{n-1}, \xi^n_{k_n})$.    
\State Set
$x^n \leftarrow (1-\beta_n)x^0 + \beta_n\left(\frac{1}{k_n}\displaystyle{\sum_{j=1}^{k_n}} \tilde{T}(x^{n-1}, \xi^n_{j})\right)$
\EndFor
\State Return $x^N$
\end{algorithmic}
\end{algorithm}
Notice that if we define the error term at the iteration $n \geq 1$  by 
\begin{equation*}
U_n := \frac{1}{k_n}\displaystyle{\sum_{j=1}^{k_n}} \tilde{T}(x^{n-1}, \xi^n_{j}) - Tx^{n-1},
\end{equation*}
then Algorithm~\ref{alg_stochastic_halpern} can be cast in the form \eqref{halpern}. Indeed, we will refer to both procedures indistinctly when there is no confusion.

By construction, $Tx^{n-1} + U_n$ is an unbiased estimator of $Tx^{n-1}$ and may have a reduced variance compared to the error computed with a single query. 
Let us introduce the following additional notation. For all $i,n\in \NN$:
\begin{equation}    \label{notation_beta_sigma}
\left \{
\begin{aligned} 
    &\sigma_n:=\EE(\norm{U_n})\,\,\, \text{and}\,\,\,\sigma_0=\beta_0=0;\\
    & B_i^n:= \prod_{j=i}^n \beta_j, \,\,\,\text{with the convention}\,\,\,  B_{i}^n=1 \,\,\, \text{if}\,\,\,  0\leq n<i.
    \end{aligned}
    \right.
\end{equation}
For the sake of completeness, we recall how a standard argument allows to transfer the variance reduction into a bound on $\sigma_n$. 

For a given iteration $n$, let $\xi^n_{1},\ldots,\xi^n_{k_n}$ be $k_n$ independent samples of the distribution $\mathcal D_{x^{n-1}}$ and let $e_n^2=\EE (||\tilde{T}(x^{n-1}, \xi^n_{1}) - T x^{n-1}||_2^2)$ be the variance of the error term of one query. Then
\begin{equation}
\sigma_n^2=\EE(\norm{U_n})^2 \le \frac{\mu^2}{k_n^2} \EE\left(\norm{\sum \nolimits_{j=1}^{k_n} \tilde{T}(x^{n-1}, \xi^n_{j}) - Tx^{n-1} }^2_2\right) \le \frac{\mu^2}{k_n}e_n^2,
\label{minibatch}
\end{equation}
where we have use Jensen's inequality, sample independence, and the equivalence between the norms. We observe that $\EE(\norm{U_n})$ scales by a factor $\frac{\mu}{\sqrt{k_n}}$ compared to $e_n$.  

As we already discussed, there exist more sophisticated variance reduction techniques to control $\sigma_n$. However, they often necessitate additional assumptions regarding the stochastic oracle's nature or assume extra operator structure.

\section{The nonexpansive case} \label{sec:nonexpansive}

In this section, we assume that $T$ is a nonexpansive operator. As we mentioned, we are interested in the number of iterations and the overall oracle complexity required by iteration \eqref{halpern} to obtain a solution $x^N$ such that $\EE(\norm{x^N-Tx^N}) \le \varepsilon$ for a given tolerance $\varepsilon>0$.  The following condition will be crucial for the results of this section, since it allows us to write some recursive bound for the expected fixed-point residual. 
\begin{equation}\tag{$\textsc{H}_1$} \label{H1}
\text{There exists } \overline \kappa <+\infty \, \text{ such that }\, \sup \nolimits_{n \geq 0}\EE (\norm{T x^n-x^0}) \leq \overline \kappa < +\infty.
\end{equation}

\subsection{Bounds for the expected fixed-point residual}

The following proposition provides a bound for $\EE(\norm{x^n-Tx^n})$, valid for any increasing sequence of steps $(\beta_n)_n\subseteq (0,1)$. 

\vspace{1ex}
\begin{proposition} \label{prop:residual-nonexpansive}
Let $(x^n)_{n\ge 1}$ the sequence generated by the Halpern iteration \eqref{halpern}. Assume that \eqref{H1} holds for some $\overline \kappa \geq 0$. Then, for all $n\ge 1$
\[
\EE(\norm{x^n-Tx^n}) \le \overline\kappa(1-\beta_n) + \sum_{i=1}^{n} B_i^n \left(\overline\kappa(\beta_i -\beta_{i-1}) + \beta_i\sigma_i+\beta_{i-1}\sigma_{i-1} \right )+ \beta_n\sigma_n,
\]
where $B^n_i$ and $\sigma_i$ are defined in \eqref{notation_beta_sigma}.
\end{proposition}

\begin{proof}
Recall that $\sigma_0=\beta_0=0$ and let $\kappa_n= \norm{Tx^n-x^0}$. We start by observing that the (random) sequence $(P_n)_n$ defined inductively as $P_0=0$ and
\[ P_n=
\kappa_{n-1} (\beta_n-\beta_{n-1})+\beta_{n-1} P_{n-1}+\beta_n\norm{U_n}+\beta_{n-1}\norm{U_{n-1}},\,\,  n\geq 1,
\]
is such that $\norm{x^n-x^{n-1}}\leq P_n$ for all $n \in \NN$. Indeed, for the first term we write
\[
\norm{x^1-x^0}= \beta_1\norm{x^0 -Tx^0 + U_1} \leq \kappa_0 (\beta_1-\beta_0) + \beta_1 \norm{U_1}=P_1.
\]
From the definition of \eqref{halpern}, using again the triangular inequality, the nonexpansivity of $T$, and the induction hypothesis, we have that for $n \geq 2$,
\begin{equation*}
\begin{aligned}
\!\!\norm{x^n\!-\!x^{n-1}} &=\norm{(1\!-\! \beta_{n}) x^0 + \beta_n(Tx^{n-1} +U_n)-(1\!-\!\beta_{n-1})x^0 \!-\! \beta_{n-1}(Tx^{n-2} +U_{n\!-\!1})}\\
&=\norm{(\beta_{n}\!-\!\beta_{n-1})(T x^{n-1}\!-\!x^0)\! + \!\beta_{n-1}(Tx^{n-1}\!\!-Tx^{n-2}) \!+\!\beta_{n} U_n\!-\!\beta_{n-1}U_{n-1}} \\
&\leq (\beta_n-\beta_{n-1})\norm{T x^{n-1}\!-\!x^0}\!+\! \beta_{n-1}\norm{x^{n-1}\!-\!x^{n-2}}\!+\!\beta_n\norm{U_{n}}\!+\!\beta_{n-1}\norm{U_{n-1}} \\
&\le \kappa_{n-1} (\beta_n-\beta_{n-1})+\beta_{n-1} P_{n-1}+\beta_n\norm{U_{n}}\!+\!\beta_{n-1}\norm{U_{n-1}}=P_n.
\end{aligned}
\end{equation*}

Then, it follows that $P_n$ can be computed explicitly as
\[
P_n=\sum_{i=1}^{n}B_i^{n-1} \left ( \kappa_{i-1}(\beta_i -\beta_{i-1}) + \beta_i\norm{U_{i}}+\beta_{i-1}\norm{U_{i-1}} \right ).
\]
where we used the convention $B^{n-1}_n = 1$. Now, casting \eqref{halpern} as
   \[
    x^{n} - Tx^n = (1-\beta_{n})(x^0-Tx^{n}) + \beta_{n} (Tx^{n-1}-Tx^n + U_{n}), \quad \forall \, n \geq 1,
   \]
we get that 
\begin{equation}\label{eq:residual}
\begin{aligned}
\!\!\!\norm{x^n-\!Tx^n} & \!\!=\! \!\norm{(1\!-\!\beta_n)(x^0-Tx^{n})+\beta_n(Tx^{n-1}-Tx^n+U_n )} \\
&\le \kappa_n (1-\beta_n) +  \beta_n  \norm{x^n-x^{n-1}}+ \beta_n\norm{U_n} \\
&\leq \kappa_n(1-\beta_n) +  \beta_n  P_n+ \beta_n\norm{U_n},\\
& = \kappa_n(1-\!\beta_n)\! +\! \\
&+\sum_{i=1}^{n}B_i^{n}\left ( \!\kappa_{i-1}(\beta_i -\beta_{i-1})\! + \!\beta_i\norm{U_i}\!\!+\!\!\beta_{i-1}\norm{U_{i-1}} \!\right )+ \beta_n\norm{U_n}
\end{aligned}
\end{equation}
where we used that $B^n_i = \beta_nB^{n-1}_i$ for every $i=1,...,n$ and that $B^{n-1}_{n} = 1$. Taking expectations in the expression above, and using that $\sigma_i = \EE(\norm{U_i})$, we obtain
\begin{equation*}
\begin{aligned}
\EE(\norm{x^n\!\!-\!\!Tx^n}) &\leq \EE(\kappa_n)(1-\!\beta_n)+ \sum_{i=1}^{n}B_i^{n}\left ( \EE(\kappa_{i-1})(\beta_i -\beta_{i-1})\! + \!\beta_i\sigma_i +\beta_{i\!-\!1}\sigma_{i-1} \!\right )\!+\! \beta_n\sigma_n. 
\end{aligned}
\end{equation*}
The conclusion follows from the fact that $\sup_{n\geq 0} \EE(\kappa_n) \leq \overline \kappa$.
\end{proof}

We complement the result in Proposition~\ref{prop:residual-nonexpansive} with the following lemma, which provides two simple independent and sufficient conditions for \eqref{H2} to hold.

\vspace{1ex}

\begin{lemma} \label{lem:kappa}
{\em Let $T$ be nonexpansive and $(x^n)_n$ be the sequence generated by \eqref{halpern}. Recalling the notation \eqref{notation_beta_sigma}, it holds:
\begin{itemize}
    \item[$(i)$] If $T$ has bounded range i.e. there exist a constant $M>0$ such that $\norm{Tx} \le M$ for all $x\in \RR^d$, then the condition \eqref{H1} holds with $\bar \kappa = M + \norm{x^0}$.
    \item[$(ii)$] If the sequences $(\beta_n)_n$ and $(\sigma_n)_n$ satisfies
    \begin{equation}\label{H2}\tag{$\textsc{H}_2$}
    \sup_{n\ge 1}\sum_{i=1}^n B_i^n\sigma_i\leq S < +\infty,,
    \end{equation}
    for some $S\ge 0$, then condition \eqref{H1} holds with $\bar \kappa = 2\dist(x^0, \Fix(T)) + {S}$. 
    
    In particular, if $\sum_n \sigma_n <+\infty$ then  \eqref{H2} holds. 
\end{itemize}
}
\end{lemma}

\begin{proof} 
Part $(i)$ follows from the triangular inequality. For part $(ii)$, let us fix $x^*\in \Fix(T)$. Using the fact that $T$ is nonexpansive we have
\begin{align}
    \norm{x^n-x^*} & = \norm{(1-\beta_n)(x^0-x^*) + \beta_n(Tx^{n-1}-x^*) + \beta_nU_n} \nonumber \\
    & \le (1-\beta_n)\norm{x^0-x^*} + \beta_n \norm{x^{n-1}-x^*} + \beta_n\norm{U_n} \nonumber \\
    &=\sum_{i=0}^n B^n_{i+1}( \norm{x^0-x^*}(1-\beta_i)+\beta_i\norm{U_n})\nonumber  \\
    &=  \norm{x^0-x^*} + \sum_{i=1}^n B^n_i \norm{U_i}, \label{eq:bounded}
\end{align}

where we used $\sum_{i=0}^n B^n_{i+1}(1-\beta_i)=1$ and the convention $B^n_{n+1} = 1$. Taking expectations, we have
\begin{equation*}
\begin{aligned}
\EE(\norm{Tx^n-x^0})&\le \EE(\norm{Tx^n-Tx^*})+\norm{x^*-x^0}\\
&\le 2 \norm{x^*-x^0} + \sum_{i=1}^n B_i^n\sigma_i \leq  2 \norm{x^*-x^0}+ S.
\end{aligned}
\end{equation*}
Finally, $\sup_{n \geq 0}\EE(\norm{Tx^n-x^0})\leq 2\inf_{x^* \in \Fix T}\norm{x^*-x^0} +S.$
\end{proof}

The following result establishes a general convergence rate for the expected residual of the Halpern iteration in the special case where $\beta_n = \frac{n}{n+1}$. This stepsize is easy to handle and, more importantly, leads to optimal residual estimations in the noiseless case in the Hilbert setting, and near-optimal for general Banach spaces.  

\vspace{1ex}
\begin{theorem} \label{thm:rate-nonexpansive}
 Let $(x^n)_n$ be the sequence generated by the Halpern iteration~\eqref{halpern} with the stepsize $\beta_n=\frac{n}{n+1}$. Assume that \eqref{H1} holds for some $\overline \kappa \geq 0$ and that $\sigma_n= \EE(\norm{U_n})$ is such that $\sum_{n=1}^N n\sigma_n =\tilde O(1)$. Then, given $\varepsilon>0$,  $\EE(\norm{x^N-Tx^N})\le \varepsilon$ for all $N \geq \bar N= \tilde{O}(\varepsilon^{-1})$ iterations.
\end{theorem}

\begin{proof}
For this particular choice of $\beta_n$, we have that $B_i^n=\frac{i}{n+1}$ for all $1 \leq i\leq n$. Then 
\[
S = \sup_{N\geq 1}\sum_{n=1}^N B_n^N  \sigma_n= \sup_{N\geq 1}\frac{1}{N+1}\sum_{n=1}^N n \sigma_n< \infty, 
\]
given that $\sum_{n= 1}^N n \sigma_n = \tilde{O}(1)$ and we conclude that \eqref{H2} holds and \eqref{H1} is satisfied with $\bar \kappa=S+2\dist(x^0, \Fix T)$. Now, invoking the estimates given by Proposition~\ref{prop:residual-nonexpansive} 
\begin{equation} \label{bound-expected-nonexpansive} 
\begin{aligned}
 \!\!\! \EE(\norm{ x^N-Tx^N }) & \le \! \frac{\bar{\kappa}}{N\!+\!1} \!+\!  \frac{1}{N\!+\!1}\sum_{n=1}^N \left( \frac{\bar{\kappa}}{n\!+\!1} \!+\! \frac{n^2}{n\!+\!1}\sigma_n \!+\!(n\!-\!1)\sigma_{n-1} \right ) \!+\!\frac{N}{N\!+\!1}\sigma_N\\
&=\frac{1}{N\!+\!1}\left(\bar{\kappa}\sum_{n=0}^{N}\frac{1}{n\!+\!1}  \!+\!\sum_{n=1}^N\frac{n^2}{n\!+\!1}\sigma_n +\sum_{n=2}^N(n\!-\!1)\sigma_{n-1}+ N \sigma_N\right )\\
  &\le\frac{\bar{\kappa}}{N\!+\!1}\sum_{n=0}^{N}\frac{1}{n\!+\!1}+\frac{2}{N\!+\!1}\sum_{n=1}^N n\sigma_n - \frac{1}{N+1}\sum_{n=1}^N \sigma_n,
\end{aligned}
\end{equation}
where we have used in the second line that $\sigma_0=0$. We observe that the first term $\tilde{O}(1/N)$ and the second also is by hypothesis. Therefore, $\EE(\norm{x^N-Tx^N}) = \tilde{O}(1/N)$, and $\EE(\norm{ x^N-Tx^N }) \le \varepsilon$, for all $N \geq \bar N= \tilde{O}(\varepsilon^{-1})$ iterations as claimed.
\end{proof}
As we already mentioned the rate ${O}(1/N)$ has been established in the noiseless case for Hilbert and Banach spaces. We are not aware of a result that establish a rate $\tilde{O}(1/N)$ in the case of general normed spaces for the inexact setup. The convergence rate above is expressed in terms of the number of iterations. It is important to note that the number of queries to the stochastic oracle may not be proportional to the number of iterations due to the use of increasing minibatches.

The choice $\beta_n = \frac{n}{n+1}$ in our analysis results in a clean and simple bound but introduces a logarithmic factor. In the noiseless case, this logarithmic factor can be avoided by using $\beta_n = \frac{n}{n+2}$. However, in the stochastic setting, where we require $\sum_{n=1}^N n\sigma_n$ to be $\tilde{O}(1)$, the logarithmic factor inevitably arises. While imposing a stronger condition, such as $\sum_{n=1}^\infty n\sigma_n < \infty$, could eliminate this factor, doing so would necessitate larger minibatch sizes, ultimately leading to a deterioration in the convergence rate. Moreover, it is indeed possible to write a recursive scheme for the optimal choice of the stepsize $\beta_n$ as in the noiseless case (see \cite[Theorem 2]{contreras2023optimal} for precise details). This is, we can attempt to optimize the bound on $\beta_n$ in Proposition~\ref{prop:residual-nonexpansive} for given values of $\sigma_n$. Numerical experiences seem to show that the bounds improve only by a logarithmic factor.

The following result, whose proof is provided in  Appendix~\ref{apx:explicito}, establishes explicit bounds for the sample complexity when a precise estimate of $\sum_{n=1}^N n\sigma_n$ is available. Despite its simplicity, this bound is useful in various contexts discussed later, where slight concessions in the bounds, affecting only logarithmic terms, are made to keep the constants as explicit as possible.
\begin{lemma}\label{lem:sigma_explicito}
Under the assumptions of Theorem~\ref{thm:rate-nonexpansive}, let us suppose additionally that there exists $C\geq 0$ such that $\sum_{n=1}^N n\sigma_n \leq C \ln(N+1)$ for all $N \geq 1$. Let $\rho \geq 0$ such that $3\bar{\kappa} + 2C\leq \rho$. Then, for all $\varepsilon>0$, if $\varepsilon/\rho < e^{-1}$, then 
\[
\EE(\norm{x^N-Tx^N})\le \varepsilon \quad \text{for all }\,\, N \geq \mbox{$\left \lceil 2\frac{\rho}{\varepsilon} \ln\left (\frac{\rho}{\varepsilon}\right )\right \rceil$} .
\]
When $\varepsilon/\rho \geq e^{-1}$, then $N=1$ iteration is sufficient.
\end{lemma}

The following corollary provides an upper bound on the number of queries for the stochastic Halpern iteration with minibatching when the stochastic oracle has uniformly bounded variance. Due to Lemma~\ref{lem:sigma_explicito}, the constants involved can be explicitly estimated in this case. For the sake of simplicity, we assume in what follows that $\varepsilon>0$ is sufficiently small.

\vspace{-1ex}
\begin{corollary} \label{cor:bounded_variance}
 Assume that $\tilde{T}$ has uniformly bounded variance $\sigma^2$. Let $(x^n)_n$ be the sequence generated by the Halpern iteration~\eqref{halpern} with the stepsize $\beta_n=\frac{n}{n+1}$ and batch size $k_n = n^{4}$. Then, given $\varepsilon>0$, the Algorithm~\ref{alg_stochastic_halpern} finds $x^N$ such that $\EE(\norm{x^N-Tx^N})\le \varepsilon$ using at most $\tilde{O}(L^5\varepsilon^{-5})$ queries to the stochastic oracle, where $L=\max\{\mu \sigma, \dist(x^0, \Fix T)\}$.
\end{corollary}
\vspace{-2ex}
\noindent {\em Proof.} Equation~\eqref{minibatch} guarantees that 
 $\sigma_n\le \mu\sigma{n^{-2}},$ 
 which readily implies that
 $$\sum_{n=1}^N n\sigma_n\leq \mu\sigma\sum_{n=1}^N \frac{1}{n}\leq \frac{3}{2}\mu\sigma\ln(N+1).$$ 
 Recalling that for this particular choice of $\beta_n$, $B_i^n= \frac{i}{n+1}$, $i=1,\ldots,n$, we have that
 \[
\sum_{i=1}^n B_i^n\sigma_i=\frac{1}{n+1} \sum_{i=1}^n i \sigma_i\leq \mu\sigma\frac{1}{n+1}\sum_{i=1}^n \frac{1}{i} \leq \frac{\mu\sigma}{2}.
 \]
 Therefore, condition~\eqref{H2} holds with $S= \mu\sigma/2$ and, from Lemma~\ref{lem:kappa}$(ii)$,  \eqref{H1} is verified with $\bar \kappa= \mu\sigma/2 + 2 \dist(x^0, \Fix T)$.  Then, using Lemma~\ref{lem:sigma_explicito} with $C=\frac{3}{2}\mu\sigma$ implies that $\EE(\norm{x^N-Tx^N})\le \varepsilon$ for all $N$ larger that $ \bar N= \left \lceil 2\frac{\rho}{\varepsilon} \ln\left (\frac{\rho}{\varepsilon}\right )\right \rceil $, with \vspace{-1ex}
 $$\frac{9}{2}\mu\sigma+6 \dist(x^0, \Fix T)\leq 12 \max \{\mu\sigma,\dist(x^0, \Fix T)\}:=\rho.$$ Finally, as each iteration $n=1,...,\bar N$ requires $n^4$ queries to the stochastic oracle, the overall oracle complexity to achieve precision $\varepsilon>0$ is 
 $$ \qquad \qquad \qquad \qquad \qquad \qquad  \sum_{n=1}^{\bar N} k_n = \sum_{n=1}^{\bar N} n^{4} \le \frac{(\bar N+1)^5}{5}\leq \frac{32}{5} \left (\frac{\rho}{\varepsilon}\ln\left (\frac{\rho}{\varepsilon}\right ) +1\right )^5. \qquad \qquad \qquad \qquad \qquad  \,\, \QEDA$$

Although we have examined in detail the special case of uniformly bounded variances, Theorem~\ref{thm:rate-nonexpansive} also accommodates scenarios where  $\EE(||\tilde{T}(x^n,\xi)-Tx^n||)$ can grow in a controlled manner. This flexibility is important in the context of average reward MDPs, where rough estimates for the single-sample error are used to determine the minibatch size and derive overall oracle complexity bounds. The reader is referred to Remark~\ref{remark:non_unif_variance_average_Qlearning} for further comments. 
\begin{remark}
There is a simple reason behind the particular choice of $k_n$ in Corollary~\ref{cor:bounded_variance}. We want to find some $N(\varepsilon)$ such that \eqref{bound-expected-nonexpansive} is smaller than $\varepsilon$. Roughly speaking, if one assumes that $k_n= n^{a}$, then $\sigma_n=O(n^{-a/2})$ and $\sum_{n=1}^N n\sigma_n= O(N^{2-a/2})$, where we impose $a\geq 2$ in order to verify condition~\eqref{H2}. The overall oracle complexity is $\sum_{n=1}^N k_n =O(N^{1+a})$. 
For $a \geq 4$, the dominating term in \eqref{bound-expected-nonexpansive} is $\ln(N)/N$, implying that $N=\tilde O(\varepsilon^{-1})$, and leading then to a complexity of $\tilde O(\varepsilon^{-(1+a)})$. For $2 \leq a < 4$, the term $\sum_{n=1}^N n \sigma_n = \tilde{O}(N^{1-a/2})$ becomes dominant in \eqref{bound-expected-nonexpansive}. To achieve a precision of $\varepsilon$, we require $N = \tilde{O}(\varepsilon^{-2/(a-2)})$ with oracle complexity $\tilde{O}(\varepsilon^{-\frac{2(a+1)}{a-2}})$. In summary, for this type of minibatch sizes, the oracle complexity is $\tilde O(\varepsilon^{-\omega(a)})$ with
\[
\omega(a)= \begin{cases} \dfrac{2(1+a)}{a-2}& \text{  if  }\,\, 2< a\leq 4 \\[0.2cm] 1+a &   \text{  if  }\,\, a\geq 4,\end{cases}
\]
attaining its minimum at $a = 4$.
\end{remark}

\subsection{Lower bound}

In this section, we establish a lower bound on the number of queries required by any algorithm whose iterates \( x^n \) are generated as linear combinations of previous iterates and oracle evaluations of the nonexpansive map \( T \) at \( x^{n-1} \), potentially using minibatches. Our result applies to deterministic algorithms, meaning those that do not introduce additional randomness in their update steps \cite{nemirovskij1983problem}. That is, if the oracle were noiseless, the algorithm would produce a fully deterministic sequence. 

We denote by $\spann\left\{v^1, v^2,..., v^k\right\}$ the linear span of the vectors $\left\{v^1, v^2,..., v^k\right\} \subseteq \RR^d$. The construction of the operator is inspired by \cite{park2022exact}, while the stochastic oracle relies on the idea of resistant oracle, introduced in \cite{arjevani2023lower}.
% This distinguishes it from stochastic methods, such as stochastic gradient descent, where randomness plays a role in the updates.
\vspace{1ex}
\begin{theorem} \label{thm:lowerbound}
Let $\sigma, \bar \kappa >0$, $(k_n)_n\subseteq \NN$ a sequence of minibatch sizes, and $0 < \varepsilon < \frac{\sigma}{2}$ a given tolerance. Then, there exists a normed space $(\RR^d,\norm{\cdot})$ (with $d$ depending on $\varepsilon$ and $\bar \kappa$), an operator $T$ (depending on $\varepsilon$), a stochastic oracle $\tilde{T}$, an initial point $x^0\in \RR^d$ and a natural number $N = \Omega\left(\bar \kappa \sigma^2\varepsilon^{-3}\right)$ such that: 
\begin{itemize}
    \item $T$ is nonexpansive and has bounded range $\bar \kappa$, {\em i.e.} $\norm{Tx}\leq \bar \kappa$ for every $x \in \RR^d$,
    \item $\tilde{T}$ is unbiased and have uniformly bounded variance $\sigma^2$, 
    \item for any sequence $(x^n)_{n}$ such that
\begin{equation} \label{eq:span}
    x^n \in \spann\left\{x^0, x^1,..., x^{n-1}, \mbox{$\frac{1}{k_n}\sum_{j=1}^{k_n} \tilde{T}(x^{n-1}, \xi^{n}_j)$} \right\}, \quad \mbox{ for all } n \ge 1
\end{equation}
we have 
$$\EE(\norm{x^n-Tx^n})>\varepsilon, \quad \mbox{for every $n$ such that $\sum_{i=1}^n k_i \le N$}.$$ 
\end{itemize}
\end{theorem}

In order to prove the Theorem, let us first define $\lambda = 2\varepsilon$, and $d=\left\lfloor\frac{\bar \kappa}{\lambda}\right\rfloor$. We consider $\mathbb{R}^d$ endowed with the norm $\norm{x}_1 = \sum_{i=1}^d |x_i|$. Let the map $Q:\RR^d\to \RR^d$ be given by
$$Q(x_1, x_2,\dots, x_d) = (\lambda-x_d, \; x_1, x_2, \dots, x_{d-1}).$$

It is easy to check that $Q$ is nonexpansive (in fact, it is distance preserving). Moreover, $Q$ has a unique fixed-point $x^* = \left(\frac{\lambda}{2}, \; \frac{\lambda}{2}, \dots, \frac{\lambda}{2}\right)^\top \in \mathbb{R}^d$. 

Let us define $C = \{x\in \RR^d: \norm{x-x^*}_\infty\le \frac{\lambda}{2}\}$ the ball for the infinity norm centered on $x^*$ with radius $\lambda/2$. Let $P_C: \RR^d \to C$ be the function defined by $(P_C(x))_i = P_{[0, \lambda]}(x_i)$ for every $i=1,\dots,d$, where $P_{[0, \lambda]}$ is the one-dimensional projection onto the interval $[0, \lambda]$. Explicitly, $P_{[0, \lambda]}(t)$ equals $t$ if $t\in [0, \lambda]$, equals $0$ if $t<0$, and equals $\lambda$ if $t>\lambda$. 

Let $T:\RR^d\to \RR^d$ defined by $T = Q\circ P_C$. One can easily check that $T$ is nonexpansive for the norm $\norm{\cdot}_1$, and $\Fix T  = \{x^*\}$. Indeed, 
$$\norm{Tx-Ty}_1 = \norm{P_C(x)-P_C(y)}_1 = \sum_{i=1}^d |P_{[0,\lambda]}(x_i) - P_{[0, \lambda]}(y_i)| \le \sum_{i=1}^d |x_i-y_i|.$$
Here we have used the fact that $|P_{[0, \lambda]}(t)-P_{[0, \lambda]}(t')|\le |t-t'|$ for all $t\in \RR$.  

Furthermore, for every $x\in \RR^d$, we have $Tx\in C$. Indeed, let $y = P_C(x)$, then $y\in C$ and therefore
$$\norm{Tx-x^*}_\infty = \norm{Qy - x^*}_\infty = \norm{y-x^*}_\infty \le \frac{\lambda}{2}.$$
Using the triangular inequality we have 
$$\norm{Tx}_\infty\le \norm{Tx-x^*}_\infty+\norm{x^*}_\infty \le \lambda,$$
so $T$ has bounded range with $\norm{Tx}_1 \le d\lambda\leq \bar\kappa$. 

Now we will construct a stochastic oracle for $T$. For this, we define the {\it progress} of $x\in \RR^d$ similarly to \cite{arjevani2023lower},
\[
\prog(x) = \max\{i \in \{1,\ldots, d\}\;:\; |x_i|>0\}, \quad \text{and} \quad \prog(0_{\RR^d}) = 0. 
\] 
Let $\xi$ be a random variable with a Bernoulli distribution of parameter $p\in (0,1)$. We define the stochastic oracle $\tilde{T}(x,\xi)$ component-wise as
\[
(\tilde{T}(x,\xi))_i= 
\begin{cases}
(Tx)_i & \text{if} \quad i\neq \mbox{prog}(x)+1,\\[0.8ex]
\dfrac{\xi}{p}(Tx)_i  & \text{if} \quad i= \mbox{prog}(x)+1,
\end{cases}
\]
when $\mbox{prog}(x)<d$, and $\tilde{T}(x,\xi)= T(x)$ if $\mbox{prog}(x)=d$. It is easy to check that $\EE(\tilde{T}(x,\xi)) = Tx$. Now, if $\mbox{prog}(x)=d$, then $\EE\left(||\tilde{T}(x,\xi)-Tx||^2_2\right)=0$ and for all $x \in \RR^d$ with  $0\leq \mbox{prog}(x)<d$
$$\EE\left(||\tilde{T}(x,\xi)-Tx||^2_2\right) \leq ((Tx)_{\prog(x)+1})^2\EE\left(\left( 1-\frac{\xi}{p}\right)^2\right) \le  \lambda^2 \frac{1-p}{p}\leq  \frac{ \lambda^2}{p},$$
where we used that $ \norm{Tx}_{\infty}\le \lambda$. Now, taking $p = \frac{\lambda^2}{\sigma^2}$ we have that $\EE\left(||Tx-\tilde{T}(x,\xi)||^2_2\right)\le \sigma^2$ for all $x \in \RR^d$.  Finally, the following lemma will be relevant to prove the lower bound. We leave its proof to  Appendix~\ref{app:norm1bound}. 
\vspace{1ex}
\begin{lemma} \label{lemma:norm1bound}
 For all $x\in \RR^d$ such that $\prog(x) < d$, we have $\norm{x-Tx}_1\ge \lambda$.
\end{lemma}
Now we have collected all the ingredients to complete the proof. 
\begin{proof}(of Theorem~\ref{thm:lowerbound})

Let $N\in \NN$ such that  $N< \frac{d}{2p} \le N+1$. Recalling that
\mbox{$\lambda= 2\varepsilon, \quad d=\left\lceil \frac{\bar \kappa}{2\lambda}\right \rceil=\left\lceil \frac{\bar \kappa}{4\varepsilon}\right \rceil, \,\,\text{and}\,\, p=\frac{\lambda^2}{\sigma^2},$}
we obtain that $N = \Omega(\kappa \sigma^2 \varepsilon^{-3})$. 

To conclude, we will prove that if $\sum_{i=1}^n k_i \le N$ then $\EE(\norm{x^n-Tx^n})> \varepsilon$. 

Let $x^0 = 0_{\mathbb{R}^d}$, so $\text{prog}(x^0) = 0$. We observe that the stochastic oracle unveils a new coordinate in iteration $n$ if one or more of the samples $\xi^{n}_1,\dots , \xi^{n}_{k_n}$ takes the value 1. Therefore, if $\text{prog}(x^{n-1})<d$, and $x^n$ were constructed according to \eqref{eq:span}, we have $\text{prog}(x^n) = \text{prog}(x^{n-1})+1$ with probability $p_n = 1-(1-p)^{k_n}$, and $\text{prog}(x^n) = \text{prog}(x^{n-1})$ with probability $1-p_n$. As a result, the random variable $\text{prog}(x^n)$ is the sum of $n$ independent Bernoulli random variables with parameters $p_1,\dots, p_n$. In particular, the expected value of $\text{prog}(x^n)$ is given by the sum of the probabilities. Now, using the Taylor expansion and Cauchy's remainder, one can verify that $(1-p)^{k_i} \ge 1-pk_i$ for every $k_i\ge 1$. Hence, we have the estimate
$$\sum_{i=1}^n p_i = \sum_{i=1}^n 1-(1-p)^{k_i} \le \sum_{i=1}^n pk_i \le  pN < \frac{d}{2}.$$

Using Markov's inequality, we have $\mathbb{P}(\mbox{prog}(x^n) < d)\ge 1-\frac{\sum_{i=1}^n p_i}{d} > \frac{1}{2}$. Now, using Lemma~\ref{lemma:norm1bound} and Markov's inequality for the random variable $\norm{x^n-Tx^n}_1$ we have
$$\frac{1}{2}< \mathbb{P}\left(\mbox{prog}(x^n) < d\right) \leq \mathbb{P}\left(\norm{x^n-Tx^n}_1 \ge \lambda \right) \le \frac{1}{\lambda}\EE(\norm{x^n-Tx^n}_1),$$
and therefore $\varepsilon < \EE(\norm{x^n-Tx^n}_1)$. 
\end{proof}

Our lower bound depends on $\varepsilon$, $\bar \kappa$, and $\sigma$, all of which are also involved in the upper bound \eqref{bound-expected-nonexpansive}. Additionally, our results apply directly to both the Halpern iteration with minibatch discussed in this paper and the stochastic Krasnoselskii-Mann (sKM) iteration discussed in \cite{bravo2024stochastic}, which does not use variance reduction. In particular, this result reveals a gap of $\tilde{O}(\varepsilon^{-3})$ compared to the upper bound $\tilde{O}(\varepsilon^{-6})$ established for (sKM), as well as a gap of $\tilde{O}(\varepsilon^{-2})$ compared to the rate $\tilde{O}(\varepsilon^{-5})$ of the stochastic Halpern iteration. This prompts the consideration of improving the analysis or exploring alternative approaches to bridge the gap.
\begin{remark} \label{rem:minibatch}
In \cite{cai2022stochastic}, a rate of $O(\varepsilon^{-3})$ was established for a modified stochastic Halpern iteration within the Euclidean framework, specifically for monotone inclusions. This result relies on additional assumptions on the oracle $\tilde T$, requiring both a multi-point structure and a Lipschitz in expectation condition. As a result, our lower bound does not directly apply to this specialized oracle, leaving the optimality of the approach in \cite{cai2022stochastic} an open question.
\end{remark}

\subsection{Convergence of the iterates} 
Even if we are mainly interested in establishing nonasymptotic bounds for the expected error term $\EE(\norm{x^n-Tx^n})$, it is interesting to write an almost sure convergence result for the Halpern iteration \eqref{halpern} for nonexpansive operators. The proof arguments are quite standard and rely on general results available for the noiseless iteration in Banach spaces (see \cite{reich1980strong}). We refer the reader to the technical details in Appendix~\ref{apx:convergence}.
\begin{theorem}\label{thm:convergence}
Let $(x^n)_n$ be the sequence generated by \eqref{halpern} where $T$ is nonexpansive,  and $\sum_{n=1}^\infty \sigma_n < +\infty$. Assume also that the sequence $(\beta_n)_n \subseteq (0,1)$ is non-decreasing and such that
\[
\sum_{j=1}^\infty (1-\beta_j)=+\infty, \quad \text{and} \quad \sum_{j=1}^n (\beta_{j}-\beta_{j-1}) <+\infty.
\]
Then
\begin{itemize}
\item[$(a)$] The fixed-point residual $\norm{x^n-Tx^n}$ goes to zero almost surely and any accumulation point of $(x^n)_n$ is a fixed-point of $T$.
\item[$(b)$] If $\left (\RR^d,\norm{\,\,\cdot\,\,} \right ) $ is a smooth space in the sense that the limit
\[
\lim_{\lambda \to 0} \dfrac{\norm{x+
\lambda\,y} -\norm{x}}{\lambda}
\]
exists for all $x,y$ in the unit sphere, then $(x^n)_n$ converges almost surely to some $x^* \in \Fix T$.
\end{itemize}
\end{theorem}
\begin{remark}
Simple examples of smooth finite-dimensional Banach spaces covered by Theorem \ref{thm:convergence} are $\RR^d$ endowed with the $\ell^p$ norm $\norm{x}_p=(\sum_i |x_i|^p)^{1/p}$ for $1<p<\infty$.
\end{remark}
\section{The contracting case} \label{sec:contracting}

Let $T$ be a contraction with parameter $\gamma \in (0,1)$, where $x^*$ is the unique fixed-point of $T$. We are interested in the number of iterations and the overall oracle complexity required to obtain a solution $x^n$ such that $\EE(\norm{x^n-x^*}) \le \varepsilon$ for a given tolerance $\varepsilon>0$. 

Our recursive bound approach here is slightly different from the one in the nonexpansive case. We take increasing batch sizes that depend on the contraction parameter $\gamma$, allowing for a sensible decrease in the oracle complexity. The following result is the analogue of Proposition~\ref{prop:residual-nonexpansive} in the contracting setting. 

\begin{proposition}\label{prop:strong-contractive}
Let $T$ be a contraction with parameter $\gamma\in(0,1)$, $(\beta_n)_n\subseteq (0,1)$ an increasing sequence of stepsizes, and $(x^n)_{n\ge 1}$ be the sequence generated by \eqref{halpern}. Then, for every $n\geq 1$, the following estimate holds \vspace{-2ex}
\begin{equation}\label{general-strong-bound}
 \norm{x^n-x^*} \le \sum_{i=0}^n \gamma^{n-i}B_{i+1}^n((1-\beta_i)\norm{x^0-x^*} + \beta_i\norm{U_i}).
\end{equation}
\end{proposition}
\vspace{-1ex}
\begin{proof}
We proceed by induction. For $n = 1$, recalling that $U_0=0$, one has that\vspace{-1ex}
\[
\begin{aligned}
   \norm{x^1\!-\!x^*}&= \norm{(1-\beta_1)(x^0-x^*)+\beta_1(Tx^{0}-x^*)+\beta_1 U_1}\\
   &\leq \beta_1 \gamma\norm{x^0-x^*}+ (1-\beta_1)\norm{x^0-x^*} + \beta_1\norm{U_1}.
\end{aligned}
\]
Now, assuming that it holds for $n-1$, we have 
\begin{equation*}
 \begin{aligned}
\norm{x^n\!-\!x^*} & = \norm{(1-\beta_n)(x^0-x^*)+\beta_n(Tx^{n-1}-x^*)+\beta_n U_n} \\
&\le  (1-\beta_n)\norm{x^0-x^*} + \gamma \beta_n  \norm{x^{n-1}-x^*}+ \beta_n\norm{U_n} \\
&\le  (1-\beta_n)\norm{x^0-x^*} + \beta_n\norm{U_n}  + \\
&\qquad \gamma \beta_n \left( \sum_{i=0}^{n-1} \gamma^{n-1-i}B_{i+1}^{n-1}((1-\beta_i)\norm{x^0-x^*} + \beta_i\norm{U_i})\right).
\end{aligned}   
\end{equation*}
\vspace{-0ex}The result follows by reordering and noting that $\beta_n B_{i+1}^{n-1}=B_{i+1}^{n}$ and $B_{n+1}^n=1$. 
\end{proof}
\vspace{-1ex}
The next result makes explicit bounds for the expected fixed-point residual in the case where $\beta_n=\frac{n}{n+1}$ while shedding light on the appropriate minibatch size in this case.
\vspace{-1ex}
\begin{theorem}\label{thm:contracting}
Let $N\in \NN$ be fixed and $x^N$ be the output generated by the sequence \eqref{halpern} with $\beta_n = \frac{n}{n+1}$. Assume that there exists some positive constant $\sigma^2$ such that
\[\sigma_n^2 \le \frac{\mu^2\sigma^2}{n^2\gamma^{N-n}}, \quad 1\leq n\leq N,
\]
where $\sigma_n= \EE(\norm{U_n})$. Then, \[\EE(\norm{x^N-x^*}) \leq \frac{4\max\{\norm{x^0-x^*},\mu\sigma\}}{(1-\gamma)(N+1)}.
\]
\end{theorem} 
\begin{proof}
Taking expectations in \eqref{general-strong-bound} and recalling that $B_i^n= \frac{i}{n+1}$ and $\sigma_0 = 0$, we have
\begin{equation} \label{contractive-bound}
\begin{aligned}
   \EE(\norm{x^N-x^*}) &\le \frac{1}{N+1}\sum_{i=0}^N \gamma^{N-i}(i+1)\left( \frac{1}{i+1}\norm{x^0-x^*} + \frac{i}{i+1}\sigma_i \right)  \\
   &\le \frac{1}{N+1}\left(\norm{x^0-x^*}\sum_{i=0}^N \gamma^{N-i} + \mu\sigma\sum_{i=1}^N \gamma^{\frac{N-i}{2}}\right) \\
   &\le \frac{1}{N+1}\left(\frac{\norm{x^0-x^*}}{1-\gamma} + \frac{\mu\sigma}{1-\sqrt{\gamma}}\right),
\end{aligned}
\end{equation}
where the first inequality follows from Proposition~\ref{prop:strong-contractive}. The second inequality uses the bound $\sigma_i \leq \sigma i \gamma^{-\frac{N-i}{2}}$. For the two terms in the third inequality, we apply the fact that $\sum_{i=0}^N \theta^{N-i} \leq \frac{1}{1-\theta}$ with $\theta = \gamma$ and $\theta = \theta^{1/2}$, respectively. The result follows by noticing that $\frac{1}{1-\sqrt{\gamma}} \le \frac{2}{1-\gamma}$ for all $\gamma \in (0,1)$ and $\norm{x^0-x^*}+ 2\mu\sigma\leq 4\max\{\norm{x^0-x^*},\mu\sigma\}$.
\end{proof}
The following corollary serves as the counterpart to Corollary~\ref{cor:bounded_variance} as it covers the special case of uniformly bounded variance for contractive operators. 
\vspace{-1ex}
\begin{corollary} \label{col:uniform-contractive}
Let $\varepsilon>0$ be given, assume that the stochastic oracle $\tilde{T}$ has uniformly bounded variance $\sigma^2$ and let $L=\max\{\norm{x^0-x^*},\mu\sigma \}$.  Define $N  = \left\lceil\frac{4L}{\varepsilon(1-\gamma)}\right\rceil$. Then Algorithm~\ref{alg_stochastic_halpern} with minibatches of size $k_n = \lceil n^2\gamma^{N-n}\rceil$ returns a solution $x^N$ such that $\EE(\norm{x^N-x^*})\le \varepsilon$ in at most $O\left(L^2 (1-\gamma)^{-3} \varepsilon^{-2}\right)$ queries to the stochastic oracle.  
\end{corollary}

\begin{proof}
The estimate \eqref{minibatch} guarantees that $\sigma_n = \EE(\norm{U_n}) \leq \mu \frac{\sigma}{\sqrt{k_n}}$. Therefore, for minibatches of size $k_n = \lceil n^2\gamma^{N-n} \rceil$, we have $\sigma_n \le \mu\sigma(n\gamma^{\frac{N-n}{2}})^{-1}$. Theorem~\ref{thm:contracting} implies that $\EE(\norm{x^N-x^*}) \le \varepsilon$. Now, as every iteration $n=1,...,N$ requires $k_n$ queries to the stochastic oracle, the precision is attained using at most\vspace{-1ex}
\[
\begin{aligned}
\sum_{n=1}^{ N} k_n \le  \sum_{n=1}^{N} ( n^2\gamma^{{N}-n}+1) \leq \frac{{ N}^2}{1-\gamma}+N.
\end{aligned}
\]
calls to the stochastic oracle.
\end{proof}
The overall oracle complexity described in the corollary depends on several factors: tolerance, contraction parameter, initial distance to the unique fixed point, and noise variance, which may not be readily known in advance. However, in the case of the Halpern $Q$-learning algorithm for discounted reward MDPs ({\em c.f.} Section~\ref{sec:discounted}), the key parameters $\norm{x^0-x^*}$  and $\mu\sigma^2$ can be estimated as functions of $(1-\gamma)$. As a result, we obtain a more precise estimate $\tilde O(\varepsilon^{-2}(1-\gamma)^{-5})$ in that setting ({\em c.f.} Theorem~\ref{teo_discounted}).

\section{Application to MDPs} \label{sec:MDP}
Consider an infinite-horizon Markov Decision Process (MDP) with state space $\S$ and action set $\A$, both assumed to be finite. Let $p(\cdot \,|\,s,a)\in\Delta(\S)$ be the probability distribution of the next state $s' \in \S$ when one chooses the action $a\in\A$ in state $s\in\S$. The function $r: \S \times \A \to [0,1]$ is the reward function, {\em i.e.} $r(s,a)$ is the reward obtained when the process is in state $s$ and the action $a$ is chosen and we denote $r_{\max}:=\max_{(s,a)\in \S\times\A} r(s,a)$. Let $\pi : \S \to \Delta(\A)$ be a possible randomized policy in the sense that $\pi(s)$ is the distribution of the action selected when the state is $s \in \S$. We say that a policy $\pi$ is deterministic if $\pi : \S \to \A$. Notice that the policies that we consider are stationary in the sense that they do not depend on time.
\subsection{Average Reward Halpern \texorpdfstring{$Q$}{Lg}-learning} 
\vspace{-0.5ex}

In this part, we consider the so-called  average reward setting, where one looks for a policy $\pi$ that maximizes the average reward
\begin{equation*} 
V_\pi(s)= \lim_{N \to \infty}\EE_\pi \left (\dfrac{1}{N}\sum_{k=0}^{N-1} r(s_k,a
_k) \,\Big |\, s_0=s\right ),\quad \text{for all} \, s \in \S,
\end{equation*}
where the trajectory $(s_k,a_k)_k$ is induced by $\pi$.  Throughout this section, we assume that the MDP at hand is {\em weakly communicating}, that is, 
\begin{equation}\label{eq:wcom} \tag{\mbox{\sc WCom}}
 \left\{
 \begin{aligned}
&\text{There is a set $\S'\subseteq \S$ such that for all $s,s'\in \S'$ there exists a}\\
&\text{policy for which $s'$ can be accessed from $s$. Also, every state in } \\
&\S \setminus \S' \text{ is transient under any policy.}\\
\end{aligned}\right.
\end{equation}
If there is no transient states, the MDP is said to be communicating. If any deterministic policy induces a unique recurrent class (with possible transient states), the MDP is unichain. Both communicating and unichain MDPs are weakly communicating \cite{puterman2014markov}. It is well-known that under \eqref{eq:wcom}, the optimal value $v^*=\max_{\pi}V_{\pi}(s)$ does not depend on the initial state and that there exists a determinitic policy $\pi^*$ achieving its value, $V_{\pi^*}(s)=v^*$ for all $s \in \S$. For every deterministic policy $\pi$, one can always define the bias vector $$h^\pi(s)= \displaystyle\mathop{\text{C-lim}}_{N\rightarrow \infty}\EE_\pi \left (\sum_{k=0}^{N-1} (r(s_k,a_k) - V_{\pi}(s))\Big |\, s_0=s \right ),\quad \forall \, s \in \S,$$
where $\displaystyle\mathop{\text{C-lim}}$ stands for the Ces\`aro limit. Then, from \cite[Chapter 9]{puterman2014markov}, one knows that the bias vector corresponding to an optimal policy $\pi^*$, $h^*:=h^{\pi^*}$, is a solution of the following Bellman (or optimal bias) equation  
\vspace{-1ex}
\begin{equation}   \label{average-value}
h(s)=\max_{a \in \A} \sum_{s' \in \S} p(s'|s,a)\left ( r(s,a) + h(s')\right ) - v^*,\quad \forall \, s \in \S,
\end{equation}
and that any solution $h$ allows to obtain an optimal policy by taking the action that attains the max at the right-hand-side in~\eqref{average-value}. Notice that if some vector $h \in \RR^{\S}$ is a solution of \eqref{average-value}, then $h + c \, \mathbf{u}$ also is for all $c \in \RR$, where $\mathbf{u} \in \RR$ is a vector of all-ones.

To address the fact that transitions are unknown, one approach is to define the so-called $Q$-factors as
\begin{equation*} Q(s,a) = \sum_{s' \in \S} p(s'|s,a) \left( r(s,a) + h(s') \right) , \quad \forall  (s,a) \in \S \times \A.
\end{equation*}
Hence, from equation~\eqref{average-value}, $h(s') = \max_{a'\in \A} Q(s',a')$ and then the equation can be rewritten as
\begin{equation}\tag{$Q_{\text{AV}}$}\label{q_fact_av} Q(s,a) = \sum_{s' \in \S} p(s'|s,a) \left( r(s,a) + \max_{a' \in \A} Q(s',a') \right) - v^*, \quad \forall\, (s,a) \in \S \times \A. \end{equation}
Additionally (\cite[Chapter 9]{puterman2014markov}), an optimal policy $\pi$ can be obtained as $\pi(s)=\argmax_{a \in A} Q(s,a)$ for any solution $Q$ of \eqref{q_fact_av}. Observe that, as with bias functions, if $Q$ is a solution, then $Q + c\, \mathbf{e}$ is also a solution for any $c \in \RR$, where $\mathbf{e} \in \RR^{\S \times \A}$ is a matrix of ones.

Naturally, solving \eqref{q_fact_av} reduces to finding a fixed point of the Bellman operator
$\H:\RR^{\S \times \A} \to \RR^{\S \times \A}$ defined as
\vspace{-1ex}
\begin{equation}\label{function_H}
  \H Q(s,a):=  \sum_{s' \in \S} p(s'|s,a)\left ( r(s,a) + \max_{a' \in \A} Q(s',a')\right )- v^*, \quad \forall\,(s,a) \in \S \times \A.
\end{equation}

It is well known that $\H$ is nonexpansive for the infinity norm $\norm{\,\cdot\,}_\infty$. Although, at first glance, the problem of computing the $Q$-factors fits within our framework, one must address the important issue that the optimal value $v^*$ is unknown. To address this, we adopt an approach inspired by the RVI-$Q$ learning procedure, which is based on a stochastic Krasnoselskii-Mann iteration and originally introduced in \cite{abounadi2001learning} for unichain MDPs, a stronger assumption than weakly communicating MDPs. Convergence guarantees and nonasymptotic bounds on the expected error for the RVI-$Q$ learning were later established in \cite{bravo2024stochastic} through the analysis of a related stochastic KM iteration and using similar ideas as the one decribed below.

Intuitively, the approach is to define a function  $f: \RR^{\S \times \A} \to \RR$ 
such that $f(Q^n)$ serves as an estimate of the optimal value $v^*$, preventing the state variable from diverging due to the noncompactness of the set $\operatorname{Fix} \mathcal{H}$. Specifically, $f$ is choosen such that for any $c \in \RR$ and $Q \in \RR^{\S \times \A}$,  $f(Q + c \,\mathbf{e}) = f(Q) + c$, where $\mathbf{e} \in \mathbb{R}^{\mathcal{S} \times \mathcal{A}}$ is a matrix filled with ones. Typical examples of such functions are $f(Q)= \max_{(s,a) \in \S\times\A}Q(s,a)$, $f(Q)= \min_{(s,a) \in \S\times\A}Q(s,a)$, $f(Q)= \frac{1}{\dim}\sum_{(s,a) \in \S\times\A}Q(s,a)$, and $f(Q)=Q(s_0,a_0)$ for some fixed $(s_0,a_0) \in \S\times\A$.

\begin{algorithm}[h]\label{alg_Q}
\caption{Halpern $Q$-learning (Average reward)}\label{alg_av}
\begin{algorithmic}[1]
\Require Sequence $(\beta_n)_n \in (0,1)$,  $(k_n)_n$ minibatch size sequence, initial point $Q^0$, function $f$ such that  $f(Q + c \,{\bf e})= f(Q) + c$.
\For{$n=1,\ldots, N$}
\State compute  
$$\hat \H_n Q^{n-1}(s,a) := r(s,a)+ \frac{1}{k_n} \sum_{i=1}^{k_n}  \max_{a'\in \A}Q^{n-1}(s_i^n(s,a),a') - f(Q^{n-1})$$ 
where $s_i^n(s,a) \sim p(\cdot | s,a)$ for every $(s,a) \in \S\times \A$.
\State set $Q^{n} =(1-\beta_n) Q^0 + \beta_n  \hat{\H}_n Q^{n-1}$
\EndFor
\State Return $Q^N$
\end{algorithmic}
\end{algorithm}

One challenge in analyzing Algorithm~\ref{alg_av} is that the underlying operator may not be nonexpansive due to the choice of $f$. Consequently, we consider an alternative approach, which we refer to as the Benchmark Halpern $Q$-learning Algorithm~\ref{alg_av_bench}, with state variable $Q_v^n$. It is important to emphasize that Algorithm~\ref{alg_av_bench} is intended solely for theoretical analysis and is not practically implementable, as it explicitly relies on the optimal value $v^*$. Notice that the use of $f$ does not affect the overall complexity since it is evaluated once at every iteration.

\begin{algorithm}[H]
\caption{Benchmark Halpern $Q$-learning (Average reward)}\label{alg_av_bench}
\begin{algorithmic}[1]
\Require Sequence $(\beta_n)_n \in (0,1)$, initial point $ Q^0$,  $(k_n)_n$ minibatch size sequence.
\For{$n=1,\ldots, N$}
\State compute  
$$\tilde \H_n Q^{n-1}_v(s,a) := r(s,a)+ \frac{1}{k_n} \sum_{i=1}^{k_n}  \max_{a'\in \A}Q^{n-1}_v(s_i^n(s,a),a') - v^*$$ 
where $s_i^n(s,a) \sim p(\cdot | s,a)$ for every $(s,a) \in \S\times \A$.
\State set $Q^{n}_v =(1-\beta_n) Q^0 + \beta_n  \tilde \H_n Q^{n-1}_v$
\EndFor
\State Return $Q^N$
\end{algorithmic}
\end{algorithm}

Let us consider the sequences $(Q^n)_n$ and $(Q^n_v)_n$ generated by Algorithms \ref{alg_av} and \ref{alg_av_bench}, respectively, using the same samples and initialized at a common point $Q^0$. The following Lemma is crucial for what follows.

\vspace{1ex}

\begin{lemma} \label{lemma:constant_diff_Qn_Qnv}
Let ${\bf e}$ be a matrix of size $|\S|\times|\A|$ full of ones. Then, there exists a sequence of real numbers $(c_n)_{n\geq 0}$ such that for every $n\ge 0$, $Q^n - Q^n_v = c_n{\bf e}$. Moreover, for every $n\ge 0$, $\EE (\norm{\H Q^n_v-Q^n_v}_\infty) = \EE (\norm{\H Q^n-Q^n}_\infty)$, with $\H$ defined in \eqref{function_H}. 
\end{lemma}
\begin{proof}
The case $n=0$ is true by definition. Assume by induction that $Q^{n-1} -Q^{n-1}_v=c_{n-1}{\bf e}$ for some $c_{n-1}\in \RR$. Now, for every $(s,a) \in \S \times \A$ and, from the induction hypothesis and the properties of the function $f$, we have that
\[
\begin{aligned}
  & Q^{n}(s,a)\! -\! Q^{n}_v(s,a) \!=\!\beta_n \left ( \hat{\H}_n Q^{n-1}(s,a) - \tilde \H_n Q^{n-1}_v(s,a)) \right ),\\
  &= \beta_n \left (\! \frac{1}{k_n} \sum_{i=1}^{k_n}  \max_{a'\in \A}Q^{n-1}(s_i^n(s,a),a') \!-\!f(Q^{n-1})\! -\! \frac{1}{k_n} \sum_{i=1}^{k_n}  \max_{a'\in \A}Q^{n-1}_v(s_i^n(s,a),a') + v^*  \!\right ),\\  
  &=\beta_n \left (\! \frac{1}{k_n} \sum_{i=1}^{k_n}  \max_{a'\in \A}(Q_v^{n-1}(s_i^n(s,a),a')\! + \!c_{n-1}) \!-\!f(Q_v^{n-1} +c_{n-1}{\mathbf e} ) 
  - \frac{1}{k_n} \sum_{i=1}^{k_n}  \max_{a'\in \A}Q^{n-1}_v(s_i^n(s,a),a') + v^*  \!\right )\\
   &= \beta_n \left (  c_{n-1} - f(Q_v^{n-1}) -c_{n-1} + v^*  \right )= \beta_n \left (v^* - f(Q_v^{n-1})  \right ):=c_n,
\end{aligned}
\]
which is independent of $(s,a)$. Finally, using that $\H$ satisfies that $\H(Q+ c \,{\mathbf e})= \H Q+ c \,{\mathbf e}$ for all $Q \in \RR^{\S \times \A}$ and $c\in \RR$, we have $\EE (\norm{\H Q^n\!-\!Q^n}_\infty) = \EE (\norm{\H(Q_v^{n}\! +\!c_{n}{\mathbf e})\!-\!(Q_v^{n} \!+\!c_{n}{\mathbf e} )}_\infty)=\EE (\norm{\H Q^n_v\!-\!Q^n_v}_\infty).$
\end{proof}
\vspace{-1ex}
The following result establishes nonasymptotic rates for $\EE(\norm{\H Q^n-Q^n}_\infty)$ in the case where $\beta_n=\frac{n}{n+1}$. As a consequence of Lemma~\ref{lemma:constant_diff_Qn_Qnv}, it is sufficient to obtain a convergence rate for the benchmark sequence $\EE (\norm{\H Q^n_v-Q^n_v}_\infty)$. Recall that $h^* \in \RR^{\S}$ stands for an optimal bias function for weakly communicating average reward MDPs.
\vspace{-1ex}
\begin{theorem} \label{teo_average}
Let  $\varepsilon>0$ and set $\beta_n=\frac{n}{n+1}$ and minibatch sizes $k_n=\lceil n^6\ln(n+1)\rceil$. Let us assume that $\hsp \geq 1$. Then the sequence $(Q^n)_n$ generated by Algorithm \ref{alg_av} with $Q^0=0$ produces an iterate $Q^N$ such that $\EE \left ( \norm{\H Q^N-Q^N}_\infty\right )\leq \varepsilon$ with a number of calls to the sampling device of at most  $\tilde O \left ( |\S \times \A|\, \norm{h^*}_{\text{sp}}^7\varepsilon^{-7}\right )$. 
\end{theorem}
\vspace{-1ex}
The proof of Theorem~\ref{teo_average}, which is deferred to Appendix~\ref{app:average}, essentially relies on the estimations obtained earlier for general nonexpansive maps, along with the application of classical concentration inequalities to analyze Algorithm~\ref{alg_av_bench}, and concludes through Lemma~\ref{lemma:constant_diff_Qn_Qnv}. We include a logarithmic in $k_n$ in order to precisely control the terms that appear when using concentration inequalities.

Let us briefly discuss why we set $Q^0 = 0$ and assume that $\hsp \geq 1$. Regarding the first point, when working with an arbitrary initial condition $Q^0$, the bounds we derive depend essentially on the term $\max\{\norm{Q^0}_\infty, r_{\max}\}$, where $r_{\max} \leq 1$. Therefore, there is an incentive to choose an initial condition with a small infinity norm. Regarding the second point, we recall that our complexity bounds also depend on $\dist(Q^0, \Fix \H) = \dist(0, \Fix \H)\leq r_{\max}+\hsp$ (see Appendix~\ref{app:average}). Therefore, since our bound depends on some linear combination of $r_{\max}$ and $\hsp$, assuming $\hsp \geq 1$ highlights, in a sense, the most relevant part of the sample complexity and allows us to compare our results to the existing literature.
\vspace{1ex}

\begin{remark} \label{remark:non_unif_variance_average_Qlearning}
It is well known that the oracle used in the sampling process of Algorithm~\ref{alg_av}, $\tilde H(\cdot, \cdot)$, has a variance that can grow quadratically with respect to the norm of $Q$. To illustrate this, consider a trivial MDP where $\S = \{s_1, s_2\}$, $\A = \{a\}$, and all transition probabilities are equal to $1/2$. If $Q = (x,-x)$ with $x \in \mathbb{R}$, then $\H Q=(0,0)$ and $\mathbb{E} (\| \tilde{\H}(Q,\cdot) - \H Q \|_2^2) = x^2$, showing that the variance is unbounded.  The key property that enables Theorem~\ref{teo_average} is that $\mathbb{E} (\| \tilde{\H}(Q^n_v,\cdot) - \H Q^n_v \|_\infty)$ can be controlled as $C \| Q^n_v \|_\infty$ at the iteration, rather than over the entire space. This highlights the flexibility of our general result, derived from bounds obtained in the noiseless case and appropriate minibatching. 
\end{remark}

The following result show how the sample complexity established in Theorem~\ref{teo_average} can be also obtained in the form of PAC  (Probably Approximately Correct) bounds to guaranteed small errors with arbitrarily high probability. The proof follows almost exactly the same ideas and then we provide a sketch of the main argument in Appendix~\ref{app:high_proba}. 

\begin{proposition}\label{prop:high_proba}
Let us assume that $\norm{h^*}_{\text{sp}}\geq 1$, take $Q^0=0$, set a small precision $\varepsilon>0$, and $\delta \in (0,1)$. Then, for $\beta_n=\frac{n}{n+1}$, $k_n=\lceil n^6\ln(n+1) \rceil$, Algorithm \ref{alg_av} finds $Q^N$ such that 
$$\norm{\H Q^N-Q^N}_\infty \leq \varepsilon, \quad  \text{with probability at least} \quad 1-\delta,$$   with an overall sampling complexity of at most  $\tilde O \left ( |\S \times \A|\, \norm{h^*}_{\text{sp}}^7\varepsilon^{-7}\right )$, where the complexity bound contains also terms of $\ln(1/\delta)$.
\end{proposition}

\vspace{1ex}

Corollary~\ref{cor:policy_error} shows how our results can be translated to obtain $\varepsilon$-optimal policies which, as we discussed in Section~\ref{sec:related}.
\begin{corollary}\label{cor:policy_error}
Let us assume that $\norm{h^*}_{\text{sp}}\geq 1$, take $Q^0=0$, and set precision $\varepsilon>0$. Then, for $\beta_n=\frac{n}{n+1}$ and $k_n=\lceil n^6\ln(n+1) \rceil$, Algorithm \ref{alg_av} finds $Q^N$ such that the induced policy $ \pi^N:\S \to \A$, $\pi^N(s):= \argmax_{a \in \A} Q^N(s,a)$, $s \in \S$, verifies
\[
\EE( v^*-V_{\pi^N}(s) ) \leq \varepsilon, \quad \text{for all}\, s \in \S,
\]
after at most $\tilde O(|\S\times \A \norm{h^*}^7_{\text{sp}} \varepsilon^{-7})$ calls to the sampling device. Also, for $\delta \in (0,1)$, the policy $\pi^N$ is such that for all $s \in \S$ \vspace{-1ex}
\[
v^*-V_{\pi^N}(s) \leq \varepsilon, \quad \text{with probability at least}\, 1-\delta, \vspace{-1ex}
\]
with the same order of sample complexity, including terms depending on  $\ln(1/\delta)$.
\end{corollary}
\begin{proof}
The property that links the Bellman error to the induced policy error, with a proof available in \cite{lbc25}, is as follows. For any weakly communicating average reward MDP and for every \( Q \in \mathbb{R}^{\mathcal{S} \times \mathcal{A}} \), the policy defined by $\pi(s) = \arg\max Q(s,a), \quad \forall s \in \mathcal{S}$, satisfies the bound  
\[
v^* - V_\pi(s) \leq \|\mathcal{H}(Q) - Q\|_{\text{sp}}, \quad \forall s \in \mathcal{S}.
\]  
Using thar $ \norm{\H(Q)-Q}_{\text{sp}}\leq 2 \norm{\H(Q)-Q}_{\infty}$ and running Algorithm~\ref{alg_av} allows us to conclude after the application of Theorem~\ref{teo_average} and Proposition~\ref{prop:high_proba}.
\end{proof}

\vspace{-2ex}
    
\subsection{Discounted Halpern \texorpdfstring{$Q$}{Lg}-learning} \label{sec:discounted}
In this part, we consider the so-called {\em discounted} setting, where the reward for a policy $\pi$ is given by \vspace{-2ex}
\[
V_\pi(s)= \EE_\pi \left (\sum_{k=0}^\infty \gamma^k r(s_k,a_k) \,\Big |\, s_0=s\right ),\quad \text{for all} \, s \in \S,
\]
with $\gamma\in (0,1)$. The optimal value function $V^*:\S \to \RR$ is defined as $V^*(s)=\max_\pi V_\pi(s)$. It is well-known that there exists an optimal deterministic policy $\pi^*$ such that $V^*=V_{\pi^*}$ and the optimal value function $V^*$ is the unique solution to the Bellman equations
\[ 
V(s)=\max_{a \in \,A} \sum_{s' \in \S} p(s'|s,a)\left ( r(s,a) + \gamma V(s')\right ),\quad \forall \, s \in \S.
\]
The $Q$-factors in this case are defined as
$Q(s,a)= \sum_{s' \in \S} p(s'|s,a)\left ( r(s,a) + \gamma V(s')\right ),
$
and then the corresponding Bellman equation in the variable $Q$ is
\begin{equation}\label{op_cont} \tag{$Q_{\gamma}$}
Q(s,a)=\sum_{s' \in \S} p(s'|s,a)\left ( r(s,a) + \gamma \max_{a' \in \A} Q(s',a')\right ),\quad \forall \, (s,a) \in \S\times \A.
\end{equation}
In the discounted reward setting, the unique optimal value function 
 $Q^*(s,a)$ is the maximum expected reward when starting from state $s$ and taking action $a$. Any optimal policy can be obtained by selecting actions according to $\pi^*(s)=\argmax_{a \in \A}\, Q^*(s,a)$ for all $s \in \S$. For this reason, estimating $Q^*$ has been a central problem in this field. Naturally, $Q^*$ is the unique solution of $\T Q=Q$, where $\T:\RR^{\S \times \A} \to \RR^{\S \times \A}$ is the  $\gamma$-contracting  Bellman operator defined by \eqref{op_cont}.
 
 Algorithm \ref{alg_dis} is the instance of the stochastic Halpern Algorithm~\ref{alg_stochastic_halpern} in this context.

\begin{algorithm}[H]
\caption{Halpern $Q$-learning (discounted case)}\label{alg_dis}
\begin{algorithmic}[1]
\Require Sequence $(\beta_n)_n \in (0,1)$,  $(k_n)_n$ minibatch size sequence, initial point $Q^0$, discount factor 
$\gamma$, number of iterations $N$.
\For{$n=1,\ldots, N$}
\State compute  
$$\tilde \T_n Q^{n-1}(s,a) := r(s,a)+ \gamma\frac{1}{k_n} \sum_{i=1}^{k_n} \max_{a'\in \A}Q^{n-1}(s_i^n(s,a),a')$$ 
where $s_i^n(s,a) \sim p(\cdot | s,a)$ for every $(s,a) \in \S\times \A$.
\State set $Q^{n} =(1-\beta_n) Q^0 + \beta_n  \tilde \T_n Q^{n-1}$
\EndFor
\State Return $Q^N$ 
\end{algorithmic}
\end{algorithm}

Observe that the oracle involved in Algorithm~\ref{alg_dis} has uniformly bounded variance, allowing us to apply Corollary~\ref{col:uniform-contractive}. We know that $ \mu = 1$ and that $ \|Q^0 - Q^*\|_\infty \leq C (1-\gamma)^{-1} $ (see Appendix~\ref{app:discounted}). Roughly speaking, when bounding the variance in the Euclidean space, we get $\sigma^2\leq C\norm{Q}_2^2\leq C \dim \norm{Q}_\infty^2$. This leads to an estimate of the form $\sigma^2 \leq C |\mathcal{S} \times \mathcal{A}| (1-\gamma)^{-2}$. Therefore the constant $L$ involved in Corollary~\ref{col:uniform-contractive} is such that $L=O(\sqrt{\dim}(1-\gamma)^{-1})$ leading to an overall sample complexity of  
\[
O\left(\frac{|\mathcal{S} \times \mathcal{A}|^{2} r_{\max}^2}{\varepsilon^2 (1-\gamma)^5} \right),
\]  
which is worse than the rate stated in Theorem~\ref{teo_discounted} below. Consequently, we need more precise estimations to obtain the correct dependence on the dimension $ \dim $, up to logarithmic terms. The proof of the following theorem, similar to the one of Theorem~\ref{teo_average}, is presented in Appendix~\ref{app:discounted}. We assume that $\varepsilon>0$ is sufficiently small.
\vspace{-1ex}
\begin{theorem} \label{teo_discounted}
Let $\varepsilon >0$ and set $Q^0$ such that $\norm{Q^0}_\infty\leq \frac{r_{\max}}{1-\gamma}$. For $M>0$ defined in \eqref{bound_M}, let  $\rho= M\frac{r_{\max}}{(1-\gamma)^2}$, and $N=\left \lceil 2 \frac{\rho} {\varepsilon} \ln \left ( \frac{\rho}{\varepsilon} \right) \right \rceil$. Then, Algorithm \ref{alg_dis} with $\beta_n = \frac{n}{n+1}$, and $k_n=\lceil n^2\gamma^{N-n} \rceil$, produces an iterate $Q^N$ such that
\[
\EE \left ( \norm{Q^N-Q^*}_\infty\right )\leq \varepsilon,
\]
 using at most $\tilde{O}\left(\frac{|\S\times\A|r_{\max}^2}{\varepsilon^2(1-\gamma)^5} \right)$ queries to the sampling device. Also, for any $\delta \in (0,1)$, 
\[
\norm{Q^N-Q^*}_\infty\leq \varepsilon, \,\,\text{with probability at least }\,\, 1-\delta
\]
with the same complexity, where the bound contains also  terms of $\ln(1/\delta)$.
\end{theorem}
\vspace{-1ex}

\begin{remark}\label{rem:monotone}
Some final comments are in order.\hfill
\begin{itemize} 
\item The constant $M$ in Theorem~\ref{teo_discounted} is such that $M\leq \frac{2}{\ln(2)}(1+ \sqrt{8\ln(8\dim)})$.
\item The result in Theorem~\ref{teo_discounted} is essentially the same as the one obtained \cite{wainwright2019stochastic}, while analyzing the classical $Q$-leaning algorithm. This algorithm can be seen as a stochastic Krasnoselskii--Mann iteration where one sample by iteration is performed. Interestingly, our analysis does not require the monotonicity of the Bellman operator, specifically, the condition \( Q \le \bar{Q} \Rightarrow \mathcal{T}Q \le \mathcal{T}\bar{Q} \) component-wise on which the argument of \cite{wainwright2019variance} relies. Also, we provide the precise number of iterations needed to achieve the prescribed precision. As mentioned, the optimal complexity for approximate $Q$-values was later achieved through variance reduction \cite{wainwright2019variance} applied to the $Q$-learning algorithm.

\item In general, it is possible to derive an $\varepsilon$-optimal policy from an $\varepsilon$-optimal $Q$. It follows from the fact that if $\norm{Q-Q^*}_\infty \leq \varepsilon$, then 
\[
V^*(s)-V_\pi(s)\leq 2\frac{\varepsilon}{1-\gamma},
\]
where $\pi$ is the policy defined by $\pi(s)=\argmax_{a \in A}Q(s,a)$ (see for instance \cite[Corollary 2]{singh1994upper}). Then if one runs our algorithm until tolerance $\varepsilon(1-\gamma)/2$, the policy induced by the output will be $\varepsilon$-optimal. Of course, this affects the complexity by adding a factor $4(1-\gamma)^{-2}$.
\end{itemize}
\end{remark}

\noindent {\bf Acknowledgements.} We sincerely thank two anonymous referees for their insightful comments, which significantly improved the quality of this paper. Mario Bravo gratefully acknowledges the support by FONDECYT grants 1191924 and 1241805 and the Anillo grant ACT210005. Juan Pablo Contreras was supported by Postdoctoral Fondecyt Grant 3240505.

\vspace{3ex}

\bibliographystyle{ims}

%\bibliography{bibliography.bib}% common bib file
%% if required, the content of .bbl file can be included here once bbl is generated
%%\input sn-article.bbl
\newpage
\appendix
\section{Deferred proofs} \label{app:proofs}

\subsection{Proof of Lemma~\ref{lem:sigma_explicito}} \label{apx:explicito}
Assuming that $\sum_{n=1}^N n\sigma_n \leq C \ln(N+1)$, the bound \eqref{bound-expected-nonexpansive} implies that
\begin{equation*} 
\EE(\norm{ x^N-Tx^N}) \le \frac{1}{N+1}\left( 3\bar{\kappa}+2C\right)\ln(N+1) \le \underbrace{(3\bar{\kappa} + 2C)}_{\leq\rho}\frac{\ln(N+1)}{N+1}, \end{equation*}
where we have used the rough bound $3\ln(N+1)\geq \sum_{n=1}^{N+1}\frac{1}{n}$ for all $N \geq 1$, in order to keep the estimates simple at the expense of some small factor.

For $\varepsilon>0$, let us call $\tau:=\frac{\varepsilon}{\rho}>0$. It is sufficient to find some $\bar N\in \NN$ such that, for any $N \geq \overline{N}$,
\[
\frac{\ln(N+1)}{N+1} \leq \tau, \quad \text{or, equivalently,}\,\, -\tau (N+1)e^{-\tau(N+1)}\geq -\tau.
\]
First, note that if $\tau \geq e^{-1}$ then the inequality above holds for all $N \geq 0$. We assume in what follows that $\tau < e^{-1}$.

It is well-known that if we want to solve the equation $-\tau (x+1)e^{-\tau(x+1)}= -\tau$, the solution is given by the relation $-\tau(x+1) = W(-\tau)$, where $W$ is the Lambert function. Given that $\tau >0$, there is two solutions of this equation, provided by the upper and lower branches of $W$, $W_0$ and $W_{-1}$, respectively. We recall that both functions are continuous, defined in the interval $[-e^{-1},0]$, and such that $W_0(-e^{-1})=W_{-1}(-e^{-1})=-1$, $W_0$ increases to cero and $W_{-1}$ decreases to $-\infty$. This implies 
that for all $\tau \in (0,e^{-1})$, $W_0(-\tau) \in (-1,0)$ and that $W_{-1}(-\tau) \in (-\infty, -1)$. Let us pick the solution given by the lower branch, i.e. the one such that $x+1\geq \tau^{-1}$. In that case, 
$$x= -\frac{1}{\tau}W_{-1}(-\tau)-1.$$
Therefore if we set $m=\lceil -\frac{1}{\tau}W_{-1}(-\tau) -1 \rceil \in \NN$ we obtain that 
$$\frac{\ln(N+1)}{N+1} \leq \tau, \quad \text{for all } N\in \NN, N\geq m.$$
Let us show that  $m \leq \lceil 2\frac{\rho}{\varepsilon} \ln\left (\frac{\rho}{\varepsilon}\right ) \rceil= \lceil \frac{2}{\tau} \ln\left (\frac{1}{\tau}\right ) \rceil$. For that purpose, we use following lower bound for the function $W_{-1}$ (see \cite[Theorem 1]{chatzigeorgiou2013bounds}):

$$W_{-1}(-e^{-u-1}) \geq -1-\sqrt{2u}-u, \quad \text{for all} \, u>0.$$ 

Taking $u=-\ln(\tau)-1>0$, we get that $-e^{-u-1}=-\tau$ and hence
\[
\begin{aligned}
 \left \lceil -\frac{1}{\tau}W_{-1}(-\tau) -1 \right \rceil&\leq 
\left \lceil\frac{1}{\tau}\left(1+\sqrt{2(-\ln(\tau)-1)}+(-\ln(\tau)-1)\right )-1\right\rceil,\\
 & = 
\left \lceil\frac{1}{\tau}\left(\sqrt{-2(\ln(\tau)+1)}-\ln(\tau)\right )-1\right\rceil\leq \left \lceil 2\frac{1}{\tau} \ln\left (\frac{1}{\tau}\right )-1\right \rceil,
 \end{aligned}
\]
where the last inequality follows from the trivial bound $\sqrt{-2(u+1)}\leq -u$ for any $u\leq -1$. Observe that the last quantity is at least $1$ for all $\tau \in (0,e^{-1})$.

\subsection{Proof of Theorem~\ref{thm:convergence}}\label{apx:convergence}

\noindent $(a)$  From \eqref{eq:bounded}, we have that 
\begin{equation*}
    \norm{x^n-x^*} \leq  \norm{x^0-x^*} + \sum_{i=1}^n B^n_i \norm{U_i} \leq  \norm{x^0-x^*} + \sum_{i=1}^n \norm{U_i}.
\end{equation*}

Let us consider the random sequence $(\sum_{i=1}^n \norm{U_i})_n$ which, by definition, is positive, nondrecreasing and converges to (the possibly inifite) random variable $U:=\sum_{i=1}^\infty \norm{U_i}$. Using the monotone convergence theorem, we have that $\mathbb E \left( U\right )= \sum_{n=1}^\infty \sigma_n < \infty$ and then $U$ is almost surely finite. 

Hence, given that $\norm{x^n-x^*}  \leq  \norm{x^0-x^*} +U$,  $(x^n)_n$ is almost surely bounded and consequently $\kappa:=\sup_{n\geq 1}\norm{Tx^n-x^0}$ also is.

Now, applying this estimation to \eqref{eq:residual}, we obtain that
\begin{equation*}
\norm{x^n-Tx^n}\leq  \kappa(1-\beta_n) + \sum_{i=1}^{n}B_i^{n}\left ( \kappa(\beta_i -\beta_{i-1}) + \beta_i\norm{U_i}+\beta_{i-1}\norm{U_{i-1}} \right )+ \beta_n\norm{U_n},
\end{equation*}
almost surely. We aim to prove that each of the terms involved in the inequality above goes to zero almost surely.  In that case, the conclusion of part $(a)$ follows by noting that, since $(x^n)_n$ is almost surely bounded and $\norm{x^n-Tx^n}$ goes to zero, any cluster point of $(x^n)_n$ belongs to $\Fix T$.

Let us fix a trajectory on the event
$A=\left \{U <+\infty \,\,\,\text{and}\,\,\, \kappa < \infty \right \}.$ 
 Recalling that we already proved that $\mathbb P (A)=1$, all what follows will hold almost surely. 

First, given that $\lim_{n\to \infty}\beta_n=1$, one has that $\kappa (1-\beta_n)$ vanishes. Also, using that $U$ is finite, $\lim_{n\to \infty}\norm{U_n}=0$.

Let us prove now that $\sum_{i=1}^{n}B_i^{n}\norm{U_{i}}$ goes to zero as $n$ goes to infinity. For this purpose, for all $n \in \NN$ we define the function  $B_\bullet^n: \NN \to \RR$ as
\[
B_\bullet^n(i)= 
\begin{cases}
    B_i^n=\prod_{j=i}^n \beta_j, & \text{if} \quad 1\leq i \leq n \\
    0, & \text{if}\quad  n< i.
\end{cases}
\]
Observe that $B_\bullet^n(i) \leq 1$ for every $i \in \NN$, and that, from the well-known exponential bound, one has
    \[
    B_\bullet^n(i) \leq \exp(-\sum_{j=i}^n (1-\beta_j)), \quad 1 \leq i \leq n.
    \]
Hence, 
\[
\lim_{n \to \infty}B_\bullet^n(i)=0, \quad \text{for all } i \in \NN,
\]
from the fact that $\sum_{j=1}^\infty (1-\beta_j)=\infty$ and that $B_\bullet^n(i)=0$ for $i >n$. Now, let us interpret the sequence $(\norm{U_i})_{i \geq 1}$ as a finite measure $\nu$ on the set $\NN$, {\em i.e.} $\nu(\{i\})=\norm{U_i}$ for all $i \in \NN$. We have that the sequence of functions  $B_\bullet^n$ is uniformly bounded (by the constant and $\nu$-integrable function one) and convergent point-wise to the zero function. By invoking the dominated convergence theorem, we have that 
\[
  \sum_{i=1}^{n}B_i^{n}\norm{U_i}= \int_\NN B_\bullet^n d\, \nu \to 0, \quad \text{as} \quad n \to \infty.
\]

  It remains to show that both $$u_n:=\sum_{i=1}^{n}B_i^{n}\norm{U_{i-1}} \quad \text{and} \quad  v_n:= \sum_{i=1}^{n}B_i^{n} \kappa(\beta_i -\beta_{i-1})$$ vanish as $n$ goes to infinity. 
  
  For both cases, the exact same analysis can be performed to conclude using the dominated convergence theorem. Recalling that $U_0=0$ and that $\beta_0=0$, for $u_n$ the finite measure to consider is $\nu(\{i\})=\norm{U_{i-1}}$, for all $i \in \NN$, while for $v_n$ it suffices to set $\nu(\{i\})=\kappa(\beta_{i}-\beta_{i-1})$, which is finite by hypothesis.

    \vspace{3ex}

\noindent $(b)$ This part follows closely the arguments used by \cite[Theorem 3.1]{xu2002iterative} to study strong convergence in the noiseless case for general Banach spaces. For $0<t<1$, let $z^t$ the unique solution of the fixed-point equation
    \begin{equation*}
    z^t= (1-t)\,x^0 + t\, T z^t.
    \end{equation*}
    A direct application \cite[Corollary 1]{reich1980strong} shows that $z^*:=\lim_{t \to 1}z^t$ exists and is a fixed-point of $T$. Let $J:(\RR^d, \norm{\cdot}) \to (\RR^d, \norm{\cdot}_*)$ be the (standard) duality map defined as the subdifferential of the function $\frac{1}{2}\norm{\cdot }^2$, {\em i.e.}
    \[
    J(x)= \{ x^* \in \RR^d\,|\, \langle x, x^* \rangle =\norm{x}^2= \norm{x^*}_{*}^2 \},
    \]
    where $\langle \cdot , \cdot \rangle$ and $\norm{\,\,\cdot\,\,}_*$  are the duality pairing and the dual norm, respectively. Therefore, the subdifferential inequality holds $\norm{x+y}^2 \leq \norm{x}^2 + 2 \langle y, J(x+y) \rangle$ for all $x,y \in \RR^d$.

    It is well known that a smooth finite-dimensional Banach space is uniformly smooth, and then $J$ is single valued and continuous on bounded sets. Using the subdifferential inequality and following exactly the same lines as in the proof of \cite[Theorem 3.1]{xu2002iterative}, we obtain that, for all $n \geq 1$,
    \[ 
    \begin{aligned}
    \norm{x^{n}- z^t}\leq & (1+(1-t)^2)\norm{z^t - x^{n}}^2 +\norm{x^{n}-Tx^{n}}\left (2\norm{z^t-x^{n}}+\norm{x^{n}-Tx^{n}} \right )+\\
    &+ 2(1-t) \langle x^0 -z^t, J( x^{n}-z^t)\rangle,
   \end{aligned}
    \]
    and then, 
    \[
    \langle x^0 -z^t, J(z^t - x^{n})\rangle \leq \frac{(1-t)}{2}\norm{z^t - x^{n}}^2 +\frac{\norm{x^{n}-Tx^{n}}}{2(1-t)}\left (2\norm{z^t-x^{n}}+\norm{x^{n}-Tx^{n}} \right ).
    \]
    Let us fix again a trajectory on the event $A$, where we know from $(a)$ that $\lim_{n \to \infty}\norm{x^{n}-Tx^{n}}=0$. So, 
    \[
    \limsup_{n \to +\infty}\langle x^0 -z^t, J(z^t - x^{n})\rangle \leq  \limsup_{n \to +\infty} \frac{(1-t)}{2}\norm{z^t - x^{n}}^2.
    \]
    Using the fact that $(x^n)_n$ is bounded on $A$, that $J$ is continuous, and taking the limit when $t \to 1$ we obtain
    \begin{equation}\label{eq:aux}
    \limsup_{n \to +\infty}\langle x^0 -z^*, J(z^* - x^{n})\rangle\leq 0.
        \end{equation}
    Now, on the other hand, using again the subdifferential inequality
\[
 \begin{aligned}
\norm{x^n - z^*}^2 =& \norm{(1-\beta_n)(x^0-z^*) + \beta_n (T x^{n-1}  -z^* +U_n)}^2 \\
\leq& \beta_n^2\norm{T x^{n-1} - z^* + U_n}^2 + 2 (1-\beta_n)^2\langle x^0 -z^*, J( x^{n}-z^*) \rangle,\\
\leq& \beta_n \norm{x^{n-1} - z^*}^2 + \beta_n^2\left ( 2\norm{x^{n-1} - z^*}\norm{U_n} +  \norm{U_n}^2  \right )\\
& \hspace{15ex}+ 2 (1-\beta_n)\langle x^0 -z^*, J( x^{n}-z^*) \rangle.
   \end{aligned}
\]
Finally, noticing that 

$$\sum_{n=1}^\infty \beta_n^2\left ( 2\norm{x^{n-1} - z^*}\norm{U_n} +  \norm{U_n}^2  \right ) < +\infty$$  along with \eqref{eq:aux} one can readily use \cite[Lemma 2.5]{xu2002iterative} to conclude that $\norm{x_n - z^*}$ goes to zero, that is, $x_n\to z$ as $n \to \infty$ as claimed. 

\subsection{Proof of Lemma~\ref{lemma:norm1bound}} \label{app:norm1bound}
For $n<d$, let us define $\phi_n:\RR^d \to \RR$ for
$$\phi_n(x) = |\lambda-x_1| + \sum_{i=2}^{n} |P_{[0, \lambda]}(x_{i-1})-x_i| +  P_{[0, \lambda]}(x_{n}).$$

We will prove by induction over $n$ that $\phi_n(x) \ge \lambda$ for every $x\in \RR^d$. The result of the lemma follows by noticing that if $\mbox{prog}(x)<d$ then $\norm{x-Tx}_1 = \phi_{\mbox{prog}(x)}(x)$. 

For $n=1$,  $\phi_1(x) = |\lambda - x_1| + P_{[0,\lambda]}(x_1)$. We have three cases:
\begin{itemize}
    \item if $x_1\in [0,1]$ then $\phi_1(x_1) = (\lambda-x_1)+x_1 = \lambda$,
    \item if $x_1>\lambda$ then $\phi_1(x_1)= (x_1-\lambda) + \lambda = x_1 > \lambda$, 
    \item if $x_1 <0$ then $\phi_1(x_1) = (\lambda-x_1) > \lambda$,
\end{itemize}
hence, the results is valid for $n=1$. 

Assume that $\phi_{n-1}(x)\ge \lambda$ for all $x\in \RR^d$. We can write
\[
\begin{aligned}
    \phi_n(x) & = |\lambda-x_1| + \sum_{i=1}^n|P_{[0,\lambda]}(x_{i-1})-x_i| + P_{[0,\lambda]}(x_n) \\
    & = \phi_{n-1}(x) + |P_{[0,\lambda]}(x_{n-1})-x_n| + P_{[0,\lambda]}(x_n) - P_{[0,\lambda]}(x_{n-1}) \\
    & \ge \lambda + |P_{[0,\lambda]}(x_{n-1})-x_n| + P_{[0,\lambda]}(x_n) - P_{[0,\lambda]}(x_{n-1})
\end{aligned}
\]

Therefore, it is sufficient to prove that the function $\eta(z, z'):=|P_{[0,\lambda]}(z)-z'| + P_{[0,\lambda]}(z') - P_{[0,\lambda]}(z)$ is nonnegative for all $(z,z') \in \RR^2$. We analyze the following cases:

\begin{itemize}
    \item if $z\in [0,\lambda]$ then for the nonexpansiveness of $P_{[0, \lambda]}$ we have
    $$P_{[0, \lambda]}(z)-P_{[0, \lambda]}(z') \le |P_{[0, \lambda]}(z)-P_{[0, \lambda]}(z')| \le |z-z'| = |P_{[0, \lambda]}(z)-z'|,$$
    which implies $\eta(z,z') \ge 0$ for all $z'\in \RR$. 

    \item If $z>\lambda$ then $\eta(z,z') = |\lambda-z'| + P_{[0, \lambda]}(z') - \lambda$ which is nonnegative again for the nonexpansiveness of $P_{[0, \lambda]}$, namely
     $$\lambda-P_{[0, \lambda]}(z') =  |P_{[0, \lambda]}(\lambda)-P_{[0, \lambda]}(z')| \le |\lambda - z'|.$$

    \item Finally, if $z<0$ we have $\eta(z,z') = |z'| + P_{[0, \lambda]}(z')$ which is clearly nonnegative.
    
\end{itemize}

\subsection{Proof of Theorem~\ref{teo_average}} \label{app:average}

Let us analyze the Benchmark Algorithm~\ref{alg_av_bench}. Recall that $r_{\max}=\max_{(s,a)} r(s,a)$, and let $\Delta:= |r_{\max} - v^*|$. Given that $Q^0=0$, the iteration takes the simple form $Q_v^{n}=\frac{n}{n+1}\tilde \H_n Q_v^{n}$, so
\begin{align*}
   \norm{Q_v^{n}}_\infty &\le 
   \frac{n}{n+1}\norm{\tilde \H_n Q_v^{n}}_{\infty} \le \frac{n}{n+1}\left(\Delta+\norm{Q^{n-1}_v}_\infty\right).
   \label{bound-Qn_v}
\end{align*}
Iterating the last inequality until $n=1$, we find the bound 
$$\norm{Q^n_v}_\infty \le C_{n+1}:= \frac{n}{2}\Delta.$$ This shows that one can obtain a deterministic, rather rough, {\em a priori} estimate on the rate with which the state variable can increase. 

 Let $\mathcal F^0$ be the trivial $\sigma$-field and consider the natural filtration $(\FF^n)_{n \geq 0}$ of the process $(Q_v^n)_{n \geq 1}$, that is, $\FF^n= \sigma (\{\mathbf{s}^m=(s_i^m(\cdot,\cdot))_{i=1}^{k_m}, m=1,\ldots, n \})$, where $\mathbf{s}^m$ is the collection of $k_m$ samples at the $m$-th iteration used to compute $Q^m_v$ in Algorithm~\ref{alg_av_bench}.
 
 For $n \geq 2$, let $\mathcal F_0^n=\mathcal F_n$ and, for $1\leq j\leq k_n$, we consider the $\sigma$-field that includes information until the $j$-th sampling when computing $Q^n_v$, that is  $$\mathcal F_j^n= \sigma (\{\mathbf{s}^m, m=1,\ldots, n \}, (s_k^{n+1}(\cdot,\cdot))_{k=1}^{j}).$$
Observe that with this definition $\mathcal F_{k_n}^n=\mathcal{F}_0^{n+1}=\mathcal F^{n+1}$ and that if we write
\begin{align*}
U_n(s,a)
& = \sum_{i=1}^{k_n}\underbrace{\frac{1}{k_n}\left(\max_{a'\in \A}Q^{n-1}_v(s_i^n(s,a),a')-\sum_{s'\in \S}p(s'|s,a)\max_{a'\in \A}Q^{n-1}_v(s',a')\right)}_{:=X^n_i(s,a)},
\end{align*}
then $X^n_i(s,a)$ is $\FF_{i}^{n-1}$-measurable. Let $Z^n_j= \sum_{i=1}^{j}X^n_i(s,a)$, $j=1,\ldots, k_n$ and $Z^n_0=0$. By construction, $(Z_j^n)_{j=0}^{k_n}$ is a zero-mean martingale with respect to the filtration $(\mathcal F^{n-1}_j)_{j=0}^{k_n}$ where $Z^{n}_{k_n}=U_n(s,a)$. Also, for all $j=1,\ldots, k_n$,
\[
|Z^{n}_{j}- Z^{n}_{j-1}|=|X^n_j(s,a)|\le 2\frac{\norm{Q^{n-1}_v}_\infty}{k_n} \le \frac{2C_n}{k_n}.
\]
Observe that this implies that $\norm{U_n}_\infty \leq 2C_n$ for all $n \in \NN$. Hence, using the classical (double-sided) Azuma-Hoeffding inequality for $n\geq 2$ (since $U_1=0$), for any $t>0$ we have 
\begin{equation*}
\mathbb{P}(|Z_{k_n}^n(s,a)| \ge t) =\mathbb{P}(|U_n(s,a)| \ge t) \le 2\exp\left(-\frac{t^2}{2k_n(2C_nk_n^{-1})^2}\right) = 2\exp\left(-\frac{k_nt^2}{8C_n^2}\right).
\end{equation*}

Setting $t= \frac{C_n}{\sqrt{k_n}}\theta_n$, with $\theta_n=\sqrt{8\ln (2\sqrt{k_n}|\S\times \A|})$, so that
\[
2\exp\left(-\frac{k_nt^2}{8C_n^2}\right)= \frac{1}{\dim \sqrt{k_n}},
\]

we get  
\begin{equation*}
\mathbb{P}\left (|U_n(s,a)| \ge \frac{C_n}{\sqrt{k_n}}\theta_n \right ) \le \frac{1}{\sqrt{k_n}|\S\times \A|}.
\end{equation*}
Using a union bound over the state-action pairs $\S \times \A$, we obtain that 
\[
\mathbb{P}\left (\norm{U_n}_\infty \ge \frac{C_n}{\sqrt{k_n}}\theta_n \right ) \le \frac{1}{\sqrt{k_n}},
\]
or, equivalently, 
\begin{equation*}
\norm{U_n}_\infty \le \frac{C_n}{\sqrt{k_n}}\theta_n\quad\text{with probability at least} \,\, 1-\frac{1}{\sqrt{k_n}}.
\end{equation*}
Now, from the law of total expectation
\begin{equation}\label{expected_Un}
\EE(\norm{U_n}_\infty) \le \frac{C_n}{\sqrt{k_n}}\theta_n + 2\frac{C_n}{\sqrt{k_n}},
\end{equation}
where we have used that $\norm{U_n}_\infty \leq 2 C_n$ for all $n \geq 2$.

Recalling that $\sigma_n=\EE(\norm{U_n}_\infty)$, the definition of $C_n$, $\sigma_n=\EE(\norm{U_n}_\infty)$, and that $k_n = \lceil n^6 \ln(n+1) \rceil$, from \eqref{expected_Un} we obtain the estimate
$$
\begin{aligned}
n\sigma_n &\leq\frac{\frac{1}{2}(n-1)\Delta}{n^2\sqrt{\ln(n+1)}}\left ( \theta_n + 2 \right ) \leq \dfrac{\Delta}{n}\cdot  \underbrace{\sup_{m \geq 2}\frac{\theta_m + 2 }{\sqrt{\ln(m+1)}}}_{:=M<\infty},
\end{aligned}
$$
where $M$ is finite,\footnote{Numerically, we obtain that $M \leq 5.93392 \sqrt{\ln(|\S \times \A|)}$.} only depending  on $\sqrt{\ln (\dim)}$.  Notice that for $n=1$, $\sigma_1=0$ since $Q^0=0$ but we include the bound in what follows for simplicity.

So, recalling that $B_i^n=\prod_{j=1}^n\beta_j=\frac{i}{n+1}$ for our specific choice of $\beta_n$, we get 
\begin{align}
\sup_{n\ge 1}\sum_{i=1}^n B_i^n\sigma_i&= \sup_{n\ge 1}\frac{1}{n+1}\sum_{i=1}^n i \sigma_i\leq \sup_{n\ge 1}\frac{1}{n+1}\sum_{i=1}^n\frac{\Delta}{i^2}\frac{\theta_i + 2}{\sqrt{\ln(i+1)}} \leq \frac{1}{2}M\Delta.\label{ineq1}
\end{align}

Let us see now that $\dist(0, \Fix \H) \leq \hsp$. Recall that $\mathbf{u} \in \mathbb{R}$ is a vector of ones. Given that $h^*$ is a solution to the optimal bias equation \eqref{average-value}, we have that $\bar h = h^* - \min_s h^*(s) \mathbf{u} \in \mathbb{R}^{\mathcal{S}}$ is also a solution, with $\hsp = ||\bar h||_\infty$ and positive entries. By definition, $\bar Q \in \mathbb{R}^{\mathcal{S} \times \mathcal{A}}$, where $\bar Q(s,a) := \sum_{s' \in \mathcal{S}} p(s'|s,a) \left( r(s,a) + \bar h(s') \right)$, is a solution to the Bellman equation \eqref{q_fact_av}, also with positive entries. Thus, $\norm{\bar Q}_\infty = \max_{s,a} \bar Q(s,a) \leq r_{\max}+ \max_{s} \bar h(s) = r_{\max}+ \hsp$ and $\dist(0, \Fix \H) \leq 2\hsp$.

Now let us compute the constants involved in terms of the parameter $\hsp$. From \eqref{ineq1} show that the assumption of Lemma~\ref{lem:kappa}$(ii)$ hold with $S= \frac{1}{2}M\Delta$ and therefore, using that $\Delta \leq r_{\max}$, \eqref{H1} holds with 
$$ \frac{1}{2}M \Delta +2\dist(0, \Fix \H )\leq \frac{1}{2}M r_{\max} +4\hsp\leq \left(\frac{1}{2}M  +4\right)\hsp :=\overline \kappa,$$
where we have used that $r_{\max} \leq 1\leq \hsp$.
On the other hand, 
\begin{align*}
\sum_{n=1}^N n \sigma_n &\leq M\Delta\sum_{n=1}^N \frac{1}{n} \leq M r_{\max}\sum_{n=1}^N \frac{1}{n} \leq  \frac{3}{2}M \hsp \ln(N+1).
\end{align*}

 This shows that we can use Lemma~\ref{lem:sigma_explicito} with $C=\frac{3}{2} M \hsp$, we conclude that Algorithm~\ref{alg_av_bench} finds a solution $Q^N_v$ such that $\EE(\norm{{\H}Q^N_v - Q^N_v}_\infty) \leq \varepsilon$ in at most 
\begin{equation*}
N=\left \lceil 2\frac{\rho}{\varepsilon} \ln\left (\frac{\rho}{\varepsilon}\right )\right \rceil, \quad \text{with} \quad \rho= \left (\frac{9}{2}M + 12\right) \hsp
\end{equation*}
iterations and the bound continues to hold afterwards. Note that if $\varepsilon/\rho \geq e^{-1}$ then one iteration is sufficient and the complexity is $\dim$. To conclude, assuming that  $\varepsilon/\rho <e^{-1}$ we observe that for each iteration $n=1,...,N$ our algorithm uses a total of $|\S\times\A|n^6 \sqrt{\ln(n+1)}$ sample calls and then, the overall complexity is
\[
\begin{aligned}
|\S\times\A|\sum_{n=1}^N n^6 \ln(n+1)\leq  \frac{1}{7} |\S\times\A|\ln \left (\left \lceil 2\frac{\rho}{\varepsilon} \ln\left (\frac{\rho}{\varepsilon}\right )\right \rceil+1 \right ) \left (\left \lceil 2\frac{\rho}{\varepsilon} \ln\left (\frac{\rho}{\varepsilon}\right )\right \rceil+1 \right )^7.
\end{aligned}
\]

\subsection{Sketch of Proof of Proposition~\ref{prop:high_proba}} \label{app:high_proba}
Again, we analyze Algorithm~\ref{alg_av_bench} and we conserve all the notation introduced there. The key idea is that, given that we have large minibatches, it is possible to obtain a bound that holds for all $n \geq 1$ with arbitrary high probability. 

More precisely, and following exactly the lines in the proof of Theorem~\ref{teo_average}, we have that for all $t>0$
\begin{equation*}
\mathbb{P}(|U_n(s,a)| \ge t) \le 2\exp\left(-\frac{k_nt^2}{16C_n^2}\right).
\end{equation*}
Now, given $\delta \in (0,1)$, we can pick $t= $ in order to have that 

\begin{equation*}
\mathbb{P}(|U_n(s,a)| \ge ) \le \frac{\delta}{K\dim\sqrt{k_n}}:=\delta_n,
\end{equation*}
where $K= \sum_{n \geq 1}\frac{1}{\sqrt{k_n}}<+\infty$. This implies that $\sum_{n\geq 1}\delta_n=\frac{\delta}{|\S \times \A|}$.

 Therefore, using a simple union bound over the state-action pairs and then over all $n \in \NN$ we get that, with probability at least $1-\delta$
\begin{equation}\label{eq:app}
(\forall n\geq 1)\,\,\,\,\norm{U_n}_\infty \le \frac{C_n}{\sqrt{k_n}}\sqrt{8 \ln \left (\frac{2K\sqrt{k_n}|\S\times \A|}{\delta} \right )}.
\end{equation}
Therefore, we can work on the event $E$ defined by \eqref{eq:app} knowing that $\mathbb P(E) \geq 1-\delta$. The bounds that follow are exactly the same ideas developed in this paper, but estimating now directly $\norm{U_n}_\infty$ and not its expectation. For instance, from \eqref{eq:app} we can bound exactly as before the quantity $\sup_{n\ge 1}\sum_{i=1}^n B_i^n\norm{U_i}_\infty \leq S < +\infty,$ and then find $\bar \kappa= S+2\dist(Q^0,\Fix \H) $ such that $\sup_{n\in N} \norm{\H Q^n - Q^0}\leq \bar \kappa$ by following line by line the proof Lemma~\ref{lem:kappa}$(ii)$. The conclusion follows by the exact same estimations given in the proof of Theorem~\ref{teo_average} that hold now on the event $E$.
\subsection{Proof of Theorem~\ref{teo_discounted}} \label{app:discounted}
First of all, the value of $M$ in the statement of the theorem is
\begin{equation}\label{bound_M}
 M= 2\sup_{N \geq 1} \frac{1+\sqrt{8\ln\left(4|\S\times \A|(N+1)\right)}}{\ln(N+1)}.
\end{equation}
Let us write
\[
Q^n(s,a)=r(s,a)+ \sum_{i=1}^{k_n}\frac{\gamma}{k_n}\max_{a'\in \A}Q^{n-1}(s_i^n(s,a),a')\leq r_{\max} + \gamma \norm{Q^{n-1}}_\infty.
\]
A simple induction shows that if $\norm{Q^0}_\infty \le \frac{r_{\max}}{1-\gamma}$, then $\norm{Q^n}_\infty\le \frac{r_{\max}}{1-\gamma}$ for all $n\ge 1$.  Moreover, $\norm{Q^*}\leq \frac{\gamma r_{\max}}{1-\gamma}$ (\cite[Lemma 1]{wainwright2019stochastic}). Observing that 
\begin{align*}
U_n(s,a)= \sum_{i=1}^{k_n}\underbrace{\frac{\gamma}{k_n}\left(\max_{a'\in \A}Q^{n-1}(s_i^n(s,a),a')-\sum_{s'\in \S}p(s'|s,a)\max_{a'\in \A}Q^{n-1}(s',a')\right)}_{:=X_i^n(s,a)},
\end{align*}
we get that $|X_i^n(s,a)|\leq 2\frac{\gamma }{k_n}\norm{Q^{n-1}}_\infty\leq 2\frac{\gamma }{k_n}\frac{r_{\max}} {(1-\gamma)}.
$
Notice that this implies that $\norm{U_n}_\infty\leq 2\frac{\gamma r_{\max}}{(1-\gamma)}$. Recall that $ k_n=\lceil\gamma^{N-n}n^2 \rceil$ and define $n(\gamma)$ the first $n$ for which $k_n\geq 2$, i.e. $n(\gamma)= \inf \{1\leq n\leq N\,:\, k_n\geq 2\}$. If such $n$ does not exists, we set $n(\gamma)=N+1$ by convention. This means that the algorithm takes only one sample matrix at each state. As in the proof of Theorem~\ref{teo_average}, based on the Azuma-Hoeffding inequality, we have that for all $t>0$ and for all $n \geq n(\gamma)$,
\begin{equation*}
\mathbb{P}(|U_n(s,a)| \ge t) \le 2\exp\left(-\frac{t^2}{2k_n(2\gamma (1-\gamma)^{-1}r_{\max}k_n^{-1})^2 }\right) = 2\exp\left(-\frac{k_n t^2 (1-\gamma)^2}{8\gamma^2 r_{\max}^2 }\right),
\end{equation*}

Let us choose $t= \gamma(1-\gamma)^{-1}r_{\max}k_n^{-1/2}\sqrt{8 \ln (2 \dim \sqrt{k_n})}$ so that 
$$2\exp\left(-\frac{k_n t^2 (1-\gamma)^2}{8\gamma^2 r_{\max}^2 }\right)= \frac{1}{\sqrt{k_n}\dim}.$$ 
Recall that $ k_n=\lceil\gamma^{N-n}n^2 \rceil$ and observe that $0<\sqrt{8 \ln (2 \dim \sqrt{k_n})}\leq \sqrt{8 \ln (2 \dim (N+1))}:=\theta_N$ for all $n \geq n(\gamma)$. From a union bound over the state-action pairs, we obtain that 
\begin{equation*}
\norm{U_n}_\infty \le \frac{\gamma r_{\max}}{\sqrt{k_n}(1-\gamma)}\theta_N\quad\text{with probability at least} \,\, 1-\frac{1}{\sqrt{k_n}}.
\end{equation*}
 Then, from the law of total expectation, using that $\norm{U_n}_\infty\leq \frac{2\gamma r_{\max}}{1-\gamma}$, we have that for all $n \geq n(\gamma)$
\begin{equation}\label{eq:aux2}
\EE(\norm{U_n}_\infty)=\sigma_n \le \frac{\gamma r_{\max}}{\sqrt{k_n}(1-\gamma)}\theta_N+ \frac{2\gamma r_{\max}}{1-\gamma}\frac{1}{\sqrt{k_n}}\leq \frac{ 2r_{\max}}{n \gamma^{\frac{N-n}{2}}(1-\gamma)}\left (\theta_N+1\right).
\end{equation}
Observe that for all $1\leq n<n(\gamma)$ we have simply that $\EE(\norm{U_n}_\infty)\leq \frac{2\gamma r_{\max}}{1-\gamma}$ and the inequality \eqref{eq:aux2} holds when $k_n=1$ since $n\gamma^{\frac{N-n}{2}}\leq 1$.  Therefore, \eqref{contractive-bound} imply that
\begin{align*}
    \EE(\norm{Q^N- Q^*}_\infty) &\le \frac{1}{N+1} \left( \sum_{n=0}^N \gamma^{N-n} \norm{Q^*}_\infty + \sum_{n=1}^N \gamma^{\frac{N-n}{2}}n\sigma_n \right) \\
    &\le \frac{1}{N+1} \left(\frac{2r_{\max}}{(1-\gamma)^2} + \frac{2r_{\max}}{(1-\gamma)(1-\sqrt{\gamma})}\left (\theta_N+1 \right)\right) \\
    &\le \frac{4r_{\max}}{(N+1)(1-\gamma)^2}\left (\theta_N+1 \right)\\
&\leq 4M\frac{r_{\max}}{(1-\gamma)^2}\frac{\ln(N+1)}{N+1}:= \rho \frac{\ln(N+1)}{N+1}.
\end{align*}
In the second line we used that $\norm{Q^*}_\infty \leq \norm{Q^0}_\infty+\norm{Q^*}_\infty\leq 2r_{\max}/(1-\gamma)$, that $\sum_{n=1}^N \gamma^{N-n}\leq (1-\gamma)^{-1}$, and the same for the sum involving powers of $\sqrt{\gamma}$. In the third line we used that $(1-\sqrt{\gamma})^{-1}\leq 2(1-\gamma)^{-1}$ and that $\gamma \leq 1$.

Therefore, using the argument in Lemma~\ref{lem:sigma_explicito}, we have that $\EE(\norm{Q^N- Q^*}_\infty)\leq \varepsilon$ if $N = \left \lceil 2 \frac{\rho}{\varepsilon} \ln \left ( \frac{\rho}{\varepsilon} \right) \right \rceil$ and the claimed bound for $Q^N$ holds. Finally, the total sample complexity can be bounded as
$$|\S\times\A|\sum_{n=1}^N k_n \le |\S\times\A|\sum_{n=1}^N (n^2\gamma^{N-n}+1) \leq \frac{|\S\times\A|(N^2 + (1-\gamma)N)}{(1-\gamma)} = \tilde{O}\left(\frac{|\S\times\A|r_{\max}^2}{\varepsilon^2 (1-\gamma)^5}\right).$$
Finally, the result with high probability follows the same modifications described in the average reward case.

\end{document}